\documentclass[reqno]{amsart}
\usepackage[utf8]{inputenc}
\usepackage{amsmath, amsthm, amssymb}
\usepackage[margin=1in]{geometry}
\usepackage{hyperref}
\usepackage{color}
\usepackage{tikz} 
\usepackage{afterpage}
\usepackage{float}
\usepackage{caption} 
\usepackage{subcaption}
\usepackage{enumitem} 
\usepackage{longtable} 
\usepackage{comment} 
\usepackage{array} 
\usepackage{multirow}

\usetikzlibrary{shapes.geometric} 

\newtheorem{theorem}{Theorem}
\numberwithin{theorem}{subsection}
\newtheorem{lemma}[theorem]{Lemma}
\newtheorem{proposition}[theorem]{Proposition}
\newtheorem{conjecture}[theorem]{Conjecture}
\newtheorem{corollary}[theorem]{Corollary}
\newtheorem{question}[theorem]{Question}
\theoremstyle{definition}
\newtheorem{definition}[theorem]{Definition}
\newtheorem{example}[theorem]{Example}
\theoremstyle{remark}
\newtheorem*{remark}{Remark}

\newcommand{\kk}{\Bbbk}

\title{On the Decomposition of Tensor Products of Monomial Modules for Finite $2$-Groups}

\author{George Cao}
\address{
Montgomery High School\\
Skillman, NJ 08558\\
U.S.A.}
\email{georgecao88@gmail.com}
\author{Kent B. Vashaw}
\address{
Department of Mathematics\\
Massachusetts Institute of Technology\\
Cambridge, MA 02139\\
U.S.A.}
\email{kentv@mit.edu}
\date{}


\begin{document}

\begin{abstract}
    Dave Benson conjectured in 2020 that if $G$ is a finite $2$-group and $V$ is an odd-dimensional indecomposable representation of $G$ over an algebraically closed field $\kk$ of characteristic $2$, then the only odd-dimensional indecomposable summand of $V \otimes V^*$ is the trivial representation $\kk$. This would imply that a tensor power of an odd-dimensional indecomposable representation of $G$ over $\kk$ has a unique odd-dimensional summand. Benson has further conjectured that, given such a representation $V$, the function sending a positive integer $n$ to the dimension of the unique odd-dimensional indecomposable summand of $V^{\otimes n}$ is quasi-polynomial. We examine this conjecture for monomial modules, a class of graded representations for the group $\mathbb{Z}/{2^r}\mathbb{Z} \times \mathbb{Z}/{2^s}\mathbb{Z}$ which correspond to skew Young diagrams. We prove the tensor powers conjecture for several modules, giving some of the first nontrivial cases where this conjecture has been verified, and we give conjectural quasi-polynomials for a broad range of monomial modules based on computational evidence.
    \end{abstract}

\maketitle

\section{Introduction}

It is well-known that the representation theory of finite groups bifurcates, depending on the characteristic of the field that the representations are taken over. When the field $\kk$ is characteristic 0, or more generally when the characteristic does not divide the order of the finite group $G$, then all representations of $G$ over $\kk$ are semisimple, and a full understanding of the tensor products of representations for $G$ can be obtained by finitely many computations. On the other hand, when the characteristic of $\kk$ divides the order of $G$, it is often no longer possible, in many cases, to compute all the indecomposable representations, and even very basic questions about the decompositions of tensor products of $G$-representations remain a mystery. Several of these open questions will be the primary focus of this paper. 

Let $\kk$ be an algebraically closed field of characteristic $p$. Based on patterns observed from a significant body of computational evidence, Dave Benson has proposed the following conjecture~\cite[Conjecture 1.1]{benson}: 

\begin{conjecture}[Benson]\label{conj:p2div4}
Let $p=2$ and $G$ a finite $2$-group. If $V$ is an odd-dimensional indecomposable representation of $\kk G$, then $V \otimes V^*$ is a direct sum of $\kk$ and indecomposable representations whose dimensions are divisible by $4$.
\end{conjecture}

The weaker version of this conjecture, in which the indecomposable representations other than $\kk$ have even dimension, is also unproven. It is well-known that a summand isomorphic to $\kk$ always exists in the decomposition for $V \otimes V^*$~\cite[Exercise 7.4]{localrep}, and indeed has multiplicity 1~\cite[Theorem 2.1]{benson-carlson}. 

If true, Conjecture~\ref{conj:p2div4} (or its weaker analogue) would also imply that any tensor power $V^{\otimes n}$ of $V$ has a unique odd-dimensional indecomposable summand. Denote this summand by $V_n$; Benson introduced the function $P_V: n \mapsto \dim(V_n)$ in~\cite[Section 1.4]{bensonbook}. He made the following conjecture:
\begin{conjecture}[Benson]\label{conj:tensorpowers}
The function $P_V(n)$ is quasi-polynomial, that is, there exist polynomials\\$f_0, f_1,\dots,$ $f_{m-1}$ such that $P_V(n)=f_i(n)$ if $n \equiv i\pmod{m}$. 
\end{conjecture}

This conjecture is closely related to work of Benson and Symonds~\cite{bensontensorpower}, where a similar conjecture was made for the dimension of the largest non-projective summand of a tensor power. In the case where $V$ is not necessarily taken to be indecomposable, questions related to the number of indecomposable summands of $V^{\otimes n}$ of dimension coprime to $p$ have been recently approached via the theory of symmetric tensor categories, see \cite[Section 8.3]{CEO} and \cite{COT}. Benson's conjectures may be formulated in a tensor-categorical context via the notion of semisimplification of a tensor category~\cite{EO1}, although such a formulation does not play a role in the present paper.

Benson's tensor powers conjecture (Conjecture~\ref{conj:tensorpowers}) is unknown even in the simplest nontrivial examples~\cite[Remark 5.8]{EtingofKannan}. We contribute by analyzing the tensor powers conjecture for a class of representations called monomial representations, which correspond to skew Young diagrams, for a certain group scheme $\alpha(r,s)$, whose group algebra is isomorphic (as an algebra, although not as a coalgebra) to the group algebra of $\mathbb{Z}/{2^r}\mathbb{Z} \times \mathbb{Z}/{2^s}\mathbb{Z}$ over $\kk$. In particular, we prove the following:

\begin{theorem}[See Corollary~\ref{4-1-poly}, Corollary~\ref{3-1-1-poly}, Corollary~\ref{4-2--1-poly}, and Corollary~\ref{6-1-poly}]
Let $\kk$ be an algebraically closed field of characteristic $2$.
\begin{enumerate}
\item Let $V$ be the monomial representation for $\alpha(1,2)$ corresponding to the Young diagram
\[
\begin{tikzpicture}[scale=0.34]
    \draw (15,-2) -- (17,-2);
    \draw (15,-1) -- (17,-1);
    \draw (15,0) -- (16,0);
    \draw (15,1) -- (16,1);
    \draw (15,2) -- (16,2);

    \draw (15,-2) -- (15,2);
    \draw (16,-2) -- (16,2);
    \draw (17,-2) -- (17,-1);
\end{tikzpicture} .
\]
Then $V^{\otimes n}$ has a unique odd-dimensional indecomposable summand, and the dimension is given by $4n+1.$
\item Let $V$ be the monomial representation for $\alpha(2,2)$ corresponding to the Young diagram
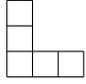
\[
\begin{tikzpicture}[scale=0.34]
    \draw (0,0) -- (3,0);
    \draw (0,1) -- (3,1);
    \draw (0,2) -- (1,2);
    \draw (0,3) -- (1,3);
    
    \draw (0,0) -- (0,3);
    \draw (1,0) -- (1,3);
    \draw (2,0) -- (2,1);
    \draw (3,0) -- (3,1);
\end{tikzpicture}.
\]
Then $V^{\otimes n}$ has a unique odd-dimensional indecomposable summand, and the dimension is given by $10n-5$ for odd $n$ and $6n+1$ for even $n$.
\item Let $V$ be the monomial representation for $\alpha(1,2)$ corresponding to the skew Young diagram
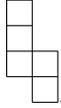
\[
\begin{tikzpicture}[scale=0.34]
    \draw (1,0) -- (2,0);
    \draw (0,1) -- (2,1);
    \draw (0,2) -- (2,2);
    \draw (0,3) -- (1,3);
    \draw (0,4) -- (1,4);
    
    \draw (0,1) -- (0,4);
    \draw (1,0) -- (1,4);
    \draw (2,0) -- (2,2);
\end{tikzpicture}.
\]
Then $V^{\otimes n}$ has a unique odd-dimensional indecomposable summand, and the dimension is given by $6n-1$ for odd $n$ and $6n+1$ for even $n$.
\item Let $V$ be the monomial representation for $\alpha(1,3)$ corresponding to the skew Young diagram
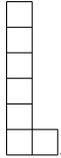
\[
\begin{tikzpicture}[scale=0.34]
    \draw (15,-2) -- (17,-2);
    \draw (15,-1) -- (17,-1);
    \draw (15,0) -- (16,0);
    \draw (15,1) -- (16,1);
    \draw (15,2) -- (16,2);
    \draw (15,3) -- (16,3);
    \draw (15,4) -- (16,4);

    \draw (15,-2) -- (15,4);
    \draw (16,-2) -- (16,4);
    \draw (17,-2) -- (17,-1);
\end{tikzpicture}.
\]
Then $V^{\otimes n}$ has a unique odd-dimensional indecomposable summand, and the dimension is given by $2n^2+4n+1$.
\end{enumerate}
\end{theorem}

We accomplish this by explicitly constructing the unique odd-dimensional indecomposable summand $V_n$ of $V^{\otimes n}$, by decomposing the tensor product $V \otimes V_{n-1}$. It is enough to consider this tensor product, as opposed to the full $V^{\otimes n}$, since the tensor product of an even-dimensional representation with any other representation breaks down into even-dimensional summands~\cite[Corollary 2.2]{benson}. We also provide a list of quasi-polynonomials which we conjecture, based on computational evidence, to give the function $P_V(n)$ for a number of monomial representations $V$. This list includes quasi-polynomials of period 1 and 2, and the polynomials themselves are either linear or quadratic.

The paper is structured as follows. In Section~\ref{sec:monomial}, we give background on the group schemes $\alpha(r,s)$ and the monomial representations which are the main object of study in the rest of the paper. In Section~\ref{sec:tensorpower}, we prove Benson's tensor powers conjecture (Conjecture~\ref{conj:tensorpowers}) for several monomial representations. There are no computations of this sort in the literature, and they give the first known examples where the tensor powers conjecture is verified in a way which does not follow from general homological algebra. In the process, we give an example which answers an analogue of a question asked by Benson and Symonds in~\cite{bensontensorpower}. Lastly, in Section~\ref{sec:otherconjs}, we introduce several additional questions and conjectures based on computational evidence.

\section{The group scheme $\alpha(r,s)$ and monomial representations}\label{sec:monomial}

Throughout this paper, unless otherwise stated, $\kk$ denotes an algebraically closed field of characteristic $2$, and $\otimes$ will mean $\otimes_{\kk}$. 

\subsection{The group scheme $\alpha(r,s)$}

The tensor product and duality operations for group representations (see~\cite[Section 4.4]{introrep}) are given by a canonical structure on the group algebra $\kk G$ of a cocommutative Hopf algebra, i.e.~its category of representations is equivalent to the representations of an affine group scheme~\cite{Montgomery, waterhouse}. In this section, we define an alternative cocommutative Hopf algebra structure on the group algebra $\kk G$, which is used to define all tensor products of representations over the rest of the paper. The benefit of using this alternate group scheme structure is that the tensor product of the representations we consider will respect a certain natural grading of these representations. However, based on computational evidence, we do conjecture that for the representations in this paper, the use of this alternative group scheme structure does not effect the dimensions of the indecomposable summands of the tensor powers.

Let $G:=\mathbb{Z}/2^r \mathbb{Z} \times \mathbb{Z}/2^s \mathbb{Z}$. Then $G$ has a minimal generating set $\{g,h\}$, where $g$ and $h$ commute and $g^{2^r}=1=h^{2^s}$. The group algebra $\kk G$ is the formal span of elements of $G$; by abuse of notation, we let $g$ and $h$ again refer to the corresponding vectors in $\kk G$. Set $x:=g+1$ and $y:=h+1$. Both $x$ and $y$ are nilpotent, since
\[
x^{2^r}=(g+1)^{2^r}=g^{2^r}+1^{2^r} = 1+1=0,
\]
and similarly for $y$, using the fact that $\kk$ is characteristic 2. Since $g$ and $h$ generate $\kk G$ as an algebra, then to define a (unital) algebra map from $\kk G$, it suffices to specify its value on $x$ and $y$.

\begin{definition}\label{def:comultiplication}
Define $\kk\alpha(r,s)$ as the Hopf algebra which is isomorphic to $\kk G$ as algebras, but where we use the \textit{comultiplication} structure $\alpha(r,s) \to \alpha(r,s) \otimes \alpha(r,s)$ defined by 
\[
x \mapsto x \otimes 1 + 1 \otimes x, \; \; y \mapsto y \otimes 1 + 1 \otimes y,
\]
and \textit{antipode} which is the identity map on both $x$ and $y$. 
\end{definition}

Since $\kk \alpha(r,s)$ is still cocommutative, its representations correspond to representations of an affine group scheme, which we denote $\alpha(r,s)$. This corresponds to a straightforward generalization of the affine group scheme $\alpha_{p}$ given in~\cite[Section 1.1]{waterhouse} and is a special case of a restricted enveloping algebra of a Lie algebra, as in~\cite[Section I.7.10]{Jantzen}. 

For the remainder of the paper, unless stated otherwise, the tensor product and dual of representations will be defined by the Hopf algebra structure given in Definition~\ref{def:comultiplication}. 

\begin{conjecture}\label{conj:newtensdiv4}
The analogues of Conjecture \ref{conj:p2div4} and Conjecture~\ref{conj:tensorpowers} hold for the group scheme $\alpha(r,s)$.
\end{conjecture}

Lastly, we define a $\mathbb{Z}^2$-grading on $\kk \alpha(r,s)$. 

\begin{definition}\label{grading}
Let $x$ and $y$ be the elements of $\kk \alpha(r,s)$ as above. Define $x$ and $y$ to have degrees $(1,0)$ and $(0,1) \in \mathbb{Z}^2$, respectively. 
\end{definition}

It is straightforward to observe that not only does this give $\kk \alpha(r,s)$ the structure of a graded algebra, but also the structure of a graded coalgebra, and in fact that of a graded Hopf algebra. Indeed, this is one of the primary motivations for using the group scheme $\alpha(r,s)$ as opposed to $\kk G$ for $G= \mathbb{Z}/2^{r}\mathbb{Z} \times \mathbb{Z}/2^{s} \mathbb{Z}$. Recall that if $H$ is a graded Hopf algebra and $V$ and $W$ are graded modules for $H$, then $V^*$ and $V \otimes W$ are also graded modules for $H$.

\subsection{Monomial representations}\label{subsec:monomialdefandbasicres}

Fix two positive integers $r$ and $s$, and denote $G:=\mathbb{Z}/2^r\mathbb{Z} \times \mathbb{Z}/2^s\mathbb{Z}$. We denote the elements $x$ and $y$ of $\kk G$ (or, what is the same, $\kk\alpha(r,s)$) as in the previous section. For the remainder of the paper, we concentrate on Conjecture~\ref{conj:newtensdiv4} for a specific class of representations of $G$, which we define now.

\begin{definition}
Choose a partition $\lambda = (\lambda_1, \dots, \lambda_n)$ such that $\lambda_1 \ge \cdots \ge \lambda_n > 0$ and a sub-partition $\mu = (\mu_1, \dots, \mu_m)$ such that $\mu_1 \ge \cdots \ge \mu_m > 0$ with $m < n$ and $\lambda_i \ge \mu_i$, for $i=1,..., m$. The pair $(\lambda, \mu)$ will be collectively referred to as a \textit{skew partition}, and will typically be denoted $\lambda/\mu$. The \textit{monomial representation} of $G$ corresponding to $\lambda/\mu$ has basis vectors $v_{i,j}$ such that $1 \le i \le n$ and $\mu_i < j \le \lambda_i$, where we set $\mu_q := 0$ for $q > m$. The action of $x$ sends $v_{i,j}$ to $v_{i+1,j}$ and the action of $y$ sends $v_{i,j}$ to $v_{i,j+1}$, if such basis elements exist, and otherwise sends them to 0. 
\end{definition}

Note that the actions of $x$ and $y$ commute. In order for the monomial representation associated to $\lambda/\mu$ to exist (i.e., to actually be a $G$-representation), we require that for any $i$, there can only be at most $2^s$ values of $j$ such that $v_{i,j}$ is a basis vector, and similarly with the roles of $i$ and $j$ reversed.

Taking $\kk \alpha(r,s)$ to be a graded Hopf algebra as in Definition \ref{grading}, it is clear that a monomial representation for $\kk\alpha(r,s)$ is a graded representation, where $v_{i,j}$ is taken to be in degree $(i,j)$. 

\begin{definition}
Let $V$ be a monomial representation of $G$. A \textit{monomial diagram} is the diagram resulting from drawing the grid box $(i,j)$ if $v_{i,j}$ is a basis vector. Each grid box is called a \textit{cell}. 
\end{definition}

In other words, the monomial diagram for the monomial representation corresponding to the skew partition $\lambda/\mu$ is the same thing as the skew Young diagram of $\lambda/\mu$. 

\begin{definition}
Let $D$ be the monomial diagram for the monomial representation $V$. Then $D$ is \textit{connected} if, for any choice of cells $B_1, B_2$, there exists a sequence of cells $B_1 = C_1, C_2, \dots, C_n = B_2$ for some $n$ such that $C_{i}$ is edge-adjacent to $C_{i+1}$.
\end{definition}

\begin{example}\label{ex:monomialdiagram}
If $G$ is $\mathbb{Z}/4\mathbb{Z} \times \mathbb{Z}/2\mathbb{Z}$, then
the monomial representation corresponding to skew partition\\ $(5,4,2,2,1,1)/(3,2)$ would have a basis vectors
$$v_{1,4},v_{1,5},v_{2,3},v_{2,4},v_{3,1},v_{3,2},v_{4,1},v_{4,2},v_{5,1},v_{6,1}.$$
Its monomial diagram is shown in Figure~\ref{fig:monomialrepex}. If $G$ were $\mathbb{Z}/2\mathbb{Z} \times \mathbb{Z}/4\mathbb{Z}$, then the monomial diagram for the representation corresponding to $(4,4,2,1)/(3,1)$ is shown in Figure~\ref{fig:monomialrepindecex}.
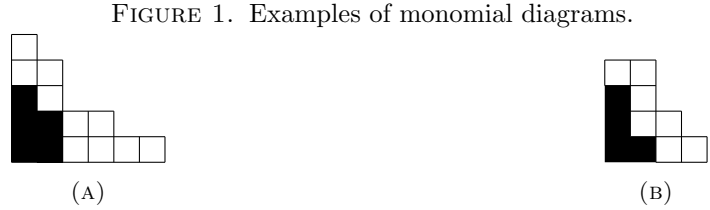
\vspace{-5mm}
\begin{figure}[H]
    \centering
    \caption{Examples of monomial diagrams.}
    \begin{subfigure}[t]{0.45\textwidth}
        \centering
        \begin{tikzpicture}[scale=0.34]
            \draw (0,0) -- (6,0);
            \draw (0,1) -- (6,1);
            \draw (0,2) -- (4,2);
            \draw (0,3) -- (2,3);
            \draw (0,4) -- (2,4);
            \draw (0,5) -- (1,5);
            
            \draw (0,0) -- (0,5);
            \draw (1,0) -- (1,5);
            \draw (2,0) -- (2,4);
            \draw (3,0) -- (3,2);
            \draw (4,0) -- (4,2);
            \draw (5,0) -- (5,1);
            \draw (6,0) -- (6,1);
            
            \fill[black] (0,0)--(1,0)--(1,3)--(0,3);
            \fill[black] (1,0)--(2,0)--(2,2)--(1,2);
        \end{tikzpicture}
        \caption{}
        \label{fig:monomialrepex}
    \end{subfigure}
    \begin{subfigure}[t]{0.45\textwidth}
        \centering
        \begin{tikzpicture}[scale=0.34]
            \draw (2,0) -- (4,0);
            \draw (1,1) -- (4,1);
            \draw (1,2) -- (3,2);
            \draw (0,3) -- (2,3);
            \draw (0,4) -- (2,4);
            \draw (0,3) -- (0,4);
            \draw (1,1) -- (1,4);
            \draw (2,0) -- (2,4);
            \draw (3,0) -- (3,2);
            \draw (4,0) -- (4,1);
            
            \fill[black] (0,0)--(2,0)--(2,1)--(1,1)--(1,3)--(0,3);
        \end{tikzpicture}
        \caption{}
        \label{fig:monomialrepindecex}
    \end{subfigure}
\end{figure}
\vspace{-5mm}
\end{example}

One may read the action of $G$ on the monomial representation $V$ directly from the monomial diagram for $V$. There is a basis of $V$ corresponding to the cells of the diagram. The action of $x$, in the diagram, is moving to the right by one cell, if the cell exists, and $0$ otherwise. The action of $y$, in the diagram, is moving up by one cell, if the cell exists, and $0$ otherwise.

A monomial representation is indecomposable if and only if its monomial diagram is connected. In Example~\ref{ex:monomialdiagram}, the representation given by Figure~\ref{fig:monomialrepex} is not indecomposable, since there are two disconnected parts: the four white boxes in the first two columns are disconnected from the 6 white boxes in columns three to six. This means that this representation can be written as a direct sum of a representation of dimension $4$ and a representation of dimension $6$. On the other hand, the monomial representations given by Figure~\ref{fig:monomialrepindecex} is indecomposable.

The following lemmas establish that the dimensions of the indecomposable summands of $V \otimes V^*$ and $V^{\otimes n}$ are not dependent on orientation of the monomial diagram for $V$, but only the shape.

\begin{lemma}\label{lem:monomialflipdiagonal}
Let $V$ be a monomial representation. Let $W$ be the monomial representation that results from reflecting its monomial diagram about the line from bottom left to top right. Then, the dimensions of the indecomposable summands of $V \otimes V^*$ are the same as the dimensions of the indecomposable summands of $W \otimes W^*$ (up to permutation), and the dimensions of the indecomposable summands of $V^{\otimes n}$ are the same as the dimensions of the indecomposable summands of $W^{\otimes n}$ (up to permutation).
\end{lemma}
\begin{proof}

First, note that if $r \leq r'$ and $s \leq s'$, we have $\kk \alpha(r,s)$ is a Hopf quotient of $\kk \alpha(r',s')$. If $V$ is a monomial representation for $\kk \alpha(r,s)$, then it is naturally also a monomial representation for $\kk \alpha(r',s')$, via this quotient map; denote this corresponding representation by $q(V)$. Since this quotient map is a map of Hopf algebras, $q(V^{\otimes n}) \cong q(V)^{\otimes n}$ and $q(V \otimes V^*) \cong q(V) \otimes q(V)^*$ and the indecomposable summands of $V^{\otimes n}$, or the indecomposable summands of $V \otimes V^*$, will have the same dimensions as the modules $q(V)^{\otimes n}$ and $q(V) \otimes q(V)^*$, for $\kk \alpha(r',s')$. 

Now, if $V$ is a monomial represention for $\alpha(r,s)$, choose $r'=s'= \max(r,s)$. Under the isomorphism $\kk\alpha(r',s') \cong \kk\alpha(s',r')$ which interchanges $x$ and $y$, the representation $q(V)$ corresponds to the representation $q(W)$, and so the dimensions of the indecomposable summands of $q(V)^{\otimes n}$ and $q(V) \otimes q(V)^*$ are equal to the dimensions of the indecomposable summands of $q(W)^{\otimes n}$ and $q(W)\otimes q(W)^*$. By the previous paragraph, these are equal, respectively, to the dimensions of the indecomposable summands of $V^{\otimes n}$ and $V \otimes V^*$. 
\end{proof}

\begin{lemma}\label{lem:monomialdual180rotate}
Let $V$ be a monomial representation. The monomial representation that results from rotating its monomial diagram $180^\circ$ gives a monomial representation isomorphic to $V^*$.
\end{lemma}
\begin{proof}

Let $V$ be a monomial representation with basis vectors $v_1, \dots, v_n$ in the cells of its monomial diagram. Let $W$ be the monomial representation resulting from a $180^\circ$ rotation of the monomial diagram of $V$. Let the basis vectors of $W$ be $w_1, \dots, w_n$ such that they are the image of $v_1, \dots, v_n$, respectively, under this rotation. 

Consider $\varphi : W \to V^*$ defined by $\varphi(w_i)v_j=1$ if $i=j$ and $0$ if $i \neq j$. This is an isomorphism of vector spaces, 
so to show that it is an isomorphism of representations we simply check that it respects the action of $G$. We have
$$\varphi(\rho_W(x)w_i)v_j = \begin{cases} 
      1 & \text{if } \rho_W(x)w_i = w_j \\
      0 & \text{otherwise},
   \end{cases}$$
$$(\rho_{V^*}(x) \cdot \varphi(w_i))v_j = \varphi(w_i) (\rho_V(x)v_j) = \begin{cases} 
      1 & \text{if } \rho_V(x)v_j = v_i \\
      0 & \text{otherwise}.
    \end{cases}$$
We have $\rho_W(x)w_i = w_j$ if and only if $w_j$ is to the right of $w_i$ (in the monomial diagram of $W$). This is true if and only if $v_j$ is to the left of $v_i$, which is true if and only if $\rho_V(x)v_j = v_i$. Thus,
$$\varphi(\rho_W(x)w_i)v_j = (\rho_{V^*}(x) \cdot \varphi(w_i))v_j$$
for all $i,j$, so $\varphi$ is a homomorphism of representations. Therefore, $\varphi$ must be an isomorphism of representations, as desired.
\end{proof}

\subsection{Computed data}\label{subsec:monomialcomputeddata}

Using the computer algebra system Magma~\cite{magma}, we can generate all possible monomial diagrams of monomial representations $V$ for a selected dimension of $V$, and compute the dimensions of the indecomposable summands of $V \otimes V^*$. Table~\ref{tab:VxV*dim3}, Table~\ref{tab:VxV*dim5}, and Table~\ref{table:VxV*dim7} contain all monomial diagrams for monomial representations of the specified dimension, along with the dimensions of the indecomposable summands of $V \otimes V^*$. Monomial diagrams that are the same shape but different orientation are omitted.

All of these examples satisfy Conjecture~\ref{conj:newtensdiv4}. No pattern between the dimensions of the indecomposable summands and the shape of the monomial diagram has been conjectured yet. 

In fact, for each example in the table below, additional data was collected using the group tensor product, and the dimensions of the indecomposable summands 
were the same as the same as using the tensor product with the comultiplication structure. We formulate this conjecture:

\begin{conjecture}
Let $V$ be an odd-dimensional indecomposable monomial representation of \\ $G=\mathbb{Z}/2^r \mathbb{Z} \times \mathbb{Z}/2^s \mathbb{Z}$. Let $\otimes$ and $(-)^*$ denote the tensor product and dual of $G$-representations defined using the Hopf algebra structure for $\kk \alpha(r,s)$, and $\otimes'$ and $(-)^\vee$ denote the tensor product and dual of $G$-representations using the standard Hopf algebra structure on $\kk G$. We conjecture that there are isomorphisms of $G$-representations
\[
V \otimes V^* \cong V \otimes' V^\vee, \quad V^{\otimes n} \cong V^{\otimes'n}.
\]
\end{conjecture}

\begin{table}[H]
\centering
\caption{Dimension $3$ monomial representations.}
\label{tab:VxV*dim3}
\begin{tabular}{|c|c|}
\hline
Dimensions & Monomial Diagrams \\\hline
\begin{tabular}{c}
$[1,4,4]$
\end{tabular}
&
\begin{tabular}{cc}
\\
\begin{tikzpicture}[scale=0.34]
    \draw (0,0) -- (1,0);
    \draw (0,1) -- (1,1);
    \draw (0,2) -- (1,2);
    \draw (0,3) -- (1,3);
    
    \draw (0,0) -- (0,3);
    \draw (1,0) -- (1,3);
\end{tikzpicture}
&
\begin{tikzpicture}[scale=0.34]
    \draw (0,0) -- (2,0);
    \draw (0,1) -- (2,1);
    \draw (0,2) -- (1,2);
    
    \draw (0,0) -- (0,2);
    \draw (1,0) -- (1,2);
    \draw (2,0) -- (2,1);
\end{tikzpicture}
\end{tabular}
\\\hline
\end{tabular}
\end{table}

\begin{table}[H]
\centering
\caption{Dimension $5$ monomial representations.}
\label{tab:VxV*dim5}
\begin{tabular}{|c|c|}
\hline
Dimensions & Monomial Diagrams\\\hline
\begin{tabular}{c}
$[1,12,12]$
\end{tabular}
&
\begin{tabular}{ccc}
\\
\begin{tikzpicture}[scale=0.34]
    \draw (0,0) -- (3,0);
    \draw (0,1) -- (3,1);
    \draw (0,2) -- (1,2);
    \draw (0,3) -- (1,3);
    
    \draw (0,0) -- (0,3);
    \draw (1,0) -- (1,3);
    \draw (2,0) -- (2,1);
    \draw (3,0) -- (3,1);
\end{tikzpicture}
&
\begin{tikzpicture}[scale=0.34]
    \draw (1,0) -- (2,0);
    \draw (0,1) -- (2,1);
    \draw (0,2) -- (2,2);
    \draw (0,3) -- (1,3);
    \draw (0,4) -- (1,4);
    
    \draw (0,1) -- (0,4);
    \draw (1,0) -- (1,4);
    \draw (2,0) -- (2,2);
\end{tikzpicture}
&
\begin{tikzpicture}[scale=0.34]
    \draw (1,0) -- (3,0);
    \draw (1,1) -- (3,1);
    \draw (0,2) -- (2,2);
    \draw (0,3) -- (2,3);
    
    \draw (0,2) -- (0,3);
    \draw (1,0) -- (1,3);
    \draw (2,0) -- (2,3);
    \draw (3,0) -- (3,1);
\end{tikzpicture}
\end{tabular}
\\\hline
\begin{tabular}{c}
$[1,4,4,8,8]$
\end{tabular}
&
\begin{tabular}{ccc}
\\
\begin{tikzpicture}[scale=0.34]
    \draw (0,0) -- (2,0);
    \draw (0,1) -- (2,1);
    \draw (0,2) -- (1,2);
    \draw (0,3) -- (1,3);
    \draw (0,4) -- (1,4);
    
    \draw (0,0) -- (0,4);
    \draw (1,0) -- (1,4);
    \draw (2,0) -- (2,1);
\end{tikzpicture}
&
\begin{tikzpicture}[scale=0.34]
    \draw (0,0) -- (2,0);
    \draw (0,1) -- (2,1);
    \draw (0,2) -- (2,2);
    \draw (0,3) -- (1,3);
    
    \draw (0,0) -- (0,3);
    \draw (1,0) -- (1,3);
    \draw (2,0) -- (2,2);
\end{tikzpicture}
&
\begin{tikzpicture}[scale=0.34]
    \draw (0,0) -- (1,0);
    \draw (0,1) -- (1,1);
    \draw (0,2) -- (1,2);
    \draw (0,3) -- (1,3);
    \draw (0,4) -- (1,4);
    \draw (0,5) -- (1,5);
    
    \draw (0,0) -- (0,5);
    \draw (1,0) -- (1,5);
\end{tikzpicture}
\end{tabular}
\\\hline
\begin{tabular}{c}
$[1,4,4,4,4,4,4]$
\end{tabular}
&
\begin{tabular}{c}
\\
\begin{tikzpicture}[scale=0.34]
    \draw (1,0) -- (3,0);
    \draw (0,1) -- (3,1);
    \draw (0,2) -- (2,2);
    \draw (0,3) -- (1,3);
    
    \draw (0,1) -- (0,3);
    \draw (1,0) -- (1,3);
    \draw (2,0) -- (2,2);
    \draw (3,0) -- (3,1);
\end{tikzpicture}
\end{tabular}
\\\hline
\end{tabular}
\end{table}

\begin{table}[H]
\centering
\caption{Dimension $7$ monomial representations.}
\label{table:VxV*dim7}
\makebox[\textwidth][c]{
\begin{tabular}{|>{\centering\arraybackslash}p{5cm}|>{\centering\arraybackslash}p{10cm}|}
\hline
Dimensions & Monomial Diagrams
\\\hline
\begin{tabular}{c}
$[1,4,4,20,20]$
\end{tabular}
&
\begin{tabular}{ccccccc}
\\
\begin{tikzpicture}[scale=0.34]
    \draw (0,0) -- (3,0);
    \draw (0,1) -- (3,1);
    \draw (0,2) -- (1,2);
    \draw (0,3) -- (1,3);
    \draw (0,4) -- (1,4);
    \draw (0,5) -- (1,5);
    \draw (0,0) -- (0,5);
    \draw (1,0) -- (1,5);
    \draw (2,0) -- (2,1);
    \draw (3,0) -- (3,1);
    
\end{tikzpicture}
&
\begin{tikzpicture}[scale=0.34]
    \draw (1,0) -- (2,0);
    \draw (0,1) -- (2,1);
    \draw (0,2) -- (2,2);
    \draw (0,3) -- (1,3);
    \draw (0,4) -- (1,4);
    \draw (0,5) -- (1,5);
    \draw (0,6) -- (1,6);
    
    \draw (0,1) -- (0,6);
    \draw (1,0) -- (1,6);
    \draw (2,0) -- (2,2);
\end{tikzpicture}
&
\begin{tikzpicture}[scale=0.34]
    \draw (1,0) -- (3,0);
    \draw (0,1) -- (3,1);
    \draw (0,2) -- (3,2);
    \draw (0,3) -- (1,3);
    \draw (0,4) -- (1,4);
    
    \draw (0,1) -- (0,4);
    \draw (1,0) -- (1,4);
    \draw (2,0) -- (2,2);
    \draw (3,0) -- (3,2);
\end{tikzpicture}
&
\begin{tikzpicture}[scale=0.34]
    \draw (2,0) -- (3,0);
    \draw (0,1) -- (3,1);
    \draw (0,2) -- (3,2);
    \draw (0,3) -- (1,3);
    \draw (0,4) -- (1,4); 
    \draw (0,5) -- (1,5);
    \draw (1,1) -- (1,5);
    \draw (2,0) -- (2,2);
    \draw (3,0) -- (3,2);
    \draw (0,1) -- (0,5);
    
\end{tikzpicture}
&
\begin{tikzpicture}[scale=0.34]
    \draw (2,0) -- (3,0);
    \draw (1,1) -- (3,1);
    \draw (0,2) -- (3,2);
    \draw (0,3) -- (2,3);
    \draw (0,4) -- (1,4);
    \draw (0,5) -- (1,5);
    
    \draw (0,2) -- (0,5);
    \draw (1,1) -- (1,5);
    \draw (2,0) -- (2,3);
    \draw (3,0) -- (3,2);
\end{tikzpicture}
&
\begin{tikzpicture}[scale=0.34]
    \draw (1,0) -- (3,0);
    \draw (1,1) -- (3,1);
    \draw (1,2) -- (2,2);
    \draw (1,3) -- (2,3);
    \draw (0,4) -- (2,4);
    \draw (0,5) -- (2,5);
    
    \draw (0,4) -- (0,5);
    \draw (1,0) -- (1,5);
    \draw (2,0) -- (2,5);
    \draw (3,0) -- (3,1);
\end{tikzpicture}
&
\begin{tikzpicture}[scale=0.34]
    \draw (2,0) -- (4,0);
    \draw (1,1) -- (4,1);
    \draw (1,2) -- (3,2);
    \draw (0,3) -- (2,3);
    \draw (0,4) -- (2,4);
    \draw (0,3) -- (0,4);
    \draw (1,1) -- (1,4);
    \draw (2,0) -- (2,4);
    \draw (3,0) -- (3,2);
    \draw (4,0) -- (4,1);
    
\end{tikzpicture}
\end{tabular}
\\\hline
\begin{tabular}{c}
$[1,8,8,16,16]$
\end{tabular}
&
\begin{tabular}{ccccc}
\\
\begin{tikzpicture}[scale=0.34]
    \draw (0,0) -- (2,0);
    \draw (0,1) -- (2,1);
    \draw (0,2) -- (1,2);
    \draw (0,3) -- (1,3);
    \draw (0,4) -- (1,4);
    \draw (0,5) -- (1,5);
    \draw (0,6) -- (1,6);
    
    \draw (0,0) -- (0,6);
    \draw (1,0) -- (1,6);
    \draw (2,0) -- (2,1);
\end{tikzpicture}
&
\begin{tikzpicture}[scale=0.34]
    \draw (0,0) -- (3,0);
    \draw (0,1) -- (3,1);
    \draw (0,2) -- (2,2);
    \draw (0,3) -- (2,3);
    
    \draw (0,0) -- (0,3);
    \draw (1,0) -- (1,3);
    \draw (2,0) -- (2,3);
    \draw (3,0) -- (3,1);
\end{tikzpicture}
&
\begin{tikzpicture}[scale=0.34]
    \draw (1,0) -- (3,0);
    \draw (0,1) -- (3,1);
    \draw (0,2) -- (2,2);
    \draw (0,3) -- (2,3);
    \draw (0,4) -- (1,4);
    
    \draw (0,1) -- (0,4);
    \draw (1,0) -- (1,4);
    \draw (2,0) -- (2,3);
    \draw (3,0) -- (3,1);
\end{tikzpicture}
&
\begin{tikzpicture}[scale=0.34]
    \draw (1,0) -- (3,0);
    \draw (1,1) -- (3,1);
    \draw (1,2) -- (2,2);
    \draw (0,3) -- (2,3);
    \draw (0,4) -- (2,4);
    \draw (0,5) -- (1,5);
    
    \draw (0,3) -- (0,5);
    \draw (1,0) -- (1,5);
    \draw (2,0) -- (2,4);
    \draw (3,0) -- (3,1);
\end{tikzpicture}
&
\begin{tikzpicture}[scale=0.34]
    \draw (2,0) -- (3,0);
    \draw (1,1) -- (3,1);
    \draw (1,2) -- (3,2);
    \draw (0,3) -- (2,3);
    \draw (0,4) -- (2,4);
    \draw (0,5) -- (1,5);
    
    \draw (0,3) -- (0,5);
    \draw (1,1) -- (1,5);
    \draw (2,0) -- (2,4);
    \draw (3,0) -- (3,2);
\end{tikzpicture}
\end{tabular}
\\\hline
\begin{tabular}{c}
$[1,24,24]$
\end{tabular}
&
\begin{tabular}{ccccc}
\\
\begin{tikzpicture}[scale=0.34]
    \draw (1,0) -- (4,0);
    \draw (0,1) -- (4,1);
    \draw (0,2) -- (2,2);
    \draw (0,3) -- (1,3);
    \draw (0,4) -- (1,4);
    
    \draw (0,1) -- (0,4);
    \draw (1,0) -- (1,4);
    \draw (2,0) -- (2,2);
    \draw (3,0) -- (3,1);
    \draw (4,0) -- (4,1);
\end{tikzpicture}
&
\begin{tikzpicture}[scale=0.34]
    \draw (1,0) -- (3,0);
    \draw (1,1) -- (3,1);
    \draw (0,2) -- (2,2);
    \draw (0,3) -- (2,3);
    \draw (0,4) -- (1,4);
    \draw (0,5) -- (1,5);
    
    \draw (0,2) -- (0,5);
    \draw (1,0) -- (1,5);
    \draw (2,0) -- (2,3);
    \draw (3,0) -- (3,1);
\end{tikzpicture}
&
\begin{tikzpicture}[scale=0.34]
    \draw (1,0) -- (4,0);
    \draw (1,1) -- (4,1);
    \draw (0,2) -- (2,2);
    \draw (0,3) -- (2,3);
    \draw (0,4) -- (1,4);
    
    \draw (0,2) -- (0,4);
    \draw (1,0) -- (1,4);
    \draw (2,0) -- (2,3);
    \draw (3,0) -- (3,1);
    \draw (4,0) -- (4,1);
\end{tikzpicture}
&
\begin{tikzpicture}[scale=0.34]
    \draw (1,0) -- (4,0);
    \draw (1,1) -- (4,1);
    \draw (1,2) -- (2,2);
    \draw (0,3) -- (2,3);
    \draw (0,4) -- (2,4);
    \draw (0,3) -- (0,4);
    \draw (1,0) -- (1,4);
    \draw (2,0) -- (2,4);
    \draw (3,0) -- (3,1);
    \draw (4,0) -- (4,1);
    
\end{tikzpicture}
&
\begin{tikzpicture}[scale=0.34]
    \draw (2,0) -- (3,0);
    \draw (2,1) -- (3,1);
    \draw (0,2) -- (3,2);
    \draw (0,3) -- (3,3);
    \draw (0,4) -- (1,4);
    \draw (0,5) -- (1,5);
    
    \draw (0,2) -- (0,5);
    \draw (1,2) -- (1,5);
    \draw (2,0) -- (2,3);
    \draw (3,0) -- (3,3);
\end{tikzpicture}
\end{tabular}
\\\hline
\begin{tabular}{c}
$[1,48]$
\end{tabular}
&
\begin{tabular}{cccc}
\\
\begin{tikzpicture}[scale=0.34]
    \draw (0,0) -- (2,0);
    \draw (0,1) -- (2,1);
    \draw (0,2) -- (2,2);
    \draw (0,3) -- (1,3);
    \draw (0,4) -- (1,4);
    \draw (0,5) -- (1,5);
    
    \draw (0,0) -- (0,5);
    \draw (1,0) -- (1,5);
    \draw (2,0) -- (2,2);
\end{tikzpicture}
&
\begin{tikzpicture}[scale=0.34]
    \draw (0,0) -- (3,0);
    \draw (0,1) -- (3,1);
    \draw (0,2) -- (2,2);
    \draw (0,3) -- (1,3);
    \draw (0,4) -- (1,4);
    
    \draw (0,0) -- (0,4);
    \draw (1,0) -- (1,4);
    \draw (2,0) -- (2,2);
    \draw (3,0) -- (3,1);
\end{tikzpicture}
&
\begin{tikzpicture}[scale=0.34]
    \draw (1,0) -- (3,0);
    \draw (0,1) -- (3,1);
    \draw (0,2) -- (3,2);
    \draw (0,3) -- (2,3);
    \draw (0,1) -- (0,3);
    \draw (1,0) -- (1,3);
    \draw (2,0) -- (2,3);
    \draw (3,0) -- (3,2);
    
\end{tikzpicture}
&
\begin{tikzpicture}[scale=0.34]
    \draw (2,0) -- (4,0);
    \draw (2,1) -- (4,1);
    \draw (0,2) -- (3,2);
    \draw (0,3) -- (3,3);
    \draw (0,4) -- (1,4);
    \draw (0,2) -- (0,4);
    \draw (1,2) -- (1,4);
    \draw (2,0) -- (2,3);
    \draw (3,0) -- (3,3);
    \draw (4,0) -- (4,1);
\end{tikzpicture}
\end{tabular}
\\\hline
\begin{tabular}{c}
$[1,8,8,8,8,8,8]$
\end{tabular}
&
\begin{tabular}{cccc}
\\
\begin{tikzpicture}[scale=0.34]
    \draw (0,0) -- (0,7);
    \draw (1,0) -- (1,7);
    \draw (0,0) -- (1,0);
    \draw (0,1) -- (1,1);
    \draw (0,2) -- (1,2);
    \draw (0,3) -- (1,3);
    \draw (0,4) -- (1,4);
    \draw (0,5) -- (1,5);
    \draw (0,6) -- (1,6);
    \draw (0,7) -- (1,7);
    
\end{tikzpicture}
&
\begin{tikzpicture}[scale=0.34]
    \draw (1,0) -- (2,0);
    \draw (0,1) -- (2,1);
    \draw (0,2) -- (2,2);
    \draw (0,3) -- (2,3);
    \draw (0,4) -- (1,4);
    \draw (0,5) -- (1,5);
    
    \draw (0,1) -- (0,5);
    \draw (1,0) -- (1,5);
    \draw (2,0) -- (2,3);
\end{tikzpicture}
&
\begin{tikzpicture}[scale=0.34]
    \draw (0,2) -- (0,4);
    \draw (1,0) -- (1,4);
    \draw (2,0) -- (2,4);
    \draw (3,0) -- (3,1);
    \draw (1,0) -- (3,0);
    \draw (1,1) -- (3,1);
    \draw (0,2) -- (2,2);
    \draw (0,3) -- (2,3);
    \draw (0,4) -- (2,4);
\end{tikzpicture}
&
\begin{tikzpicture}[scale=0.34]
    \draw (0,0) -- (2,0);
    \draw (0,1) -- (2,1);
    \draw (0,2) -- (2,2);
    \draw (0,3) -- (2,3);
    \draw (0,4) -- (1,4);
    
    \draw (0,0) -- (0,4);
    \draw (1,0) -- (1,4);
    \draw (2,0) -- (2,3);
\end{tikzpicture}
\end{tabular}
\\\hline
\end{tabular}
}
\begin{tabular}{|>{\centering\arraybackslash}p{5cm}|>{\centering\arraybackslash}p{10cm}|}
\begin{tabular}{c}
$[1,4,4,4,4,8,8,8,8]$
\end{tabular}
&
\begin{tabular}{ccc}
\\
\begin{tikzpicture}[scale=0.34]
    \draw (1,0) -- (3,0);
    \draw (0,1) -- (3,1);
    \draw (0,2) -- (2,2);
    \draw (0,3) -- (1,3);
    \draw (0,4) -- (1,4);
    \draw (0,5) -- (1,5);
    
    \draw (0,1) -- (0,5);
    \draw (1,0) -- (1,5);
    \draw (2,0) -- (2,2);
    \draw (3,0) -- (3,1);
\end{tikzpicture}
&
\begin{tikzpicture}[scale=0.34]
    \draw (1,0) -- (2,0);
    \draw (1,1) -- (2,1);
    \draw (0,2) -- (2,2);
    \draw (0,3) -- (2,3);
    \draw (0,4) -- (1,4);
    \draw (0,5) -- (1,5);
    \draw (0,6) -- (1,6);
    
    \draw (0,2) -- (0,6);
    \draw (1,0) -- (1,6);
    \draw (2,0) -- (2,3);
\end{tikzpicture}
&
\begin{tikzpicture}[scale=0.34]
    \draw (1,0) -- (3,0);
    \draw (1,1) -- (3,1);
    \draw (0,2) -- (3,2);
    \draw (0,3) -- (2,3);
    \draw (0,4) -- (1,4);
    \draw (0,2) -- (0,4);
    \draw (1,0) -- (1,4);
    \draw (2,0) -- (2,3);
    \draw (3,0) -- (3,2);
    
\end{tikzpicture}
\end{tabular}
\\\hline
\begin{tabular}{c}
$[1,4,4,4,4,16,16]$
\end{tabular}
&
\begin{tabular}{c}
\\
\begin{tikzpicture}[scale=0.34]
    \draw (0,0) -- (4,0);
    \draw (0,1) -- (4,1);
    \draw (0,2) -- (1,2);
    \draw (0,3) -- (1,3);
    \draw (0,4) -- (1,4);
    
    \draw (0,0) -- (0,4);
    \draw (1,0) -- (1,4);
    \draw (2,0) -- (2,1);
    \draw (3,0) -- (3,1);
    \draw (4,0) -- (4,1);
\end{tikzpicture}
\end{tabular}
\\\hline
\begin{tabular}{c}
$[1,4,4,40]$
\end{tabular}
&
\begin{tabular}{c}
\\
\begin{tikzpicture}[scale=0.34]
    \draw (2,0) -- (3,0);
    \draw (0,1) -- (3,1);
    \draw (0,2) -- (3,2);
    \draw (0,3) -- (2,3);
    \draw (0,4) -- (1,4);
    \draw (0,1) -- (0,4);
    \draw (1,1) -- (1,4);
    \draw (2,0) -- (2,3);
    \draw (3,0) -- (3,2);
    
\end{tikzpicture}
\end{tabular}
\\\hline
\begin{tabular}{c}
$[1,4,4,4,4,4,4,4,4,4,4,4,4]$
\end{tabular}
&
\begin{tabular}{c}
\\
\begin{tikzpicture}[scale=0.34]
    \draw (2,0) -- (4,0);
    \draw (1,1) -- (4,1);
    \draw (0,2) -- (3,2);
    \draw (0,3) -- (2,3);
    \draw (0,4) -- (1,4);
    
    \draw (0,2) -- (0,4);
    \draw (1,1) -- (1,4);
    \draw (2,0) -- (2,3);
    \draw (3,0) -- (3,2);
    \draw (4,0) -- (4,1);
\end{tikzpicture}
\end{tabular}
\\\hline
\end{tabular}
\end{table}

\section{Benson's tensor powers conjecture for monomial modules}\label{sec:tensorpower}

In this section, we consider Conjecture \ref{conj:tensorpowers} for $\alpha(r,s)$. 

\subsection{Basic results}\label{subsec:tenspowbasicresults}

We use the notation given in the introduction. Let $V$ be an odd-dimensional monomial module for $\alpha(r,s)$. We assume that Conjecture \ref{conj:p2div4} holds for $\alpha(r,s)$, so that $V^{\otimes i}$ has a unique odd-dimensional indecomposable summand, denoted by $V_i$. Recall that $P_V(n)$ denotes the dimension of $V_n$.

We may also characterize $V_i$ in the following recursive way:

\begin{proposition}
    There is a unique odd-dimensional indecomposable summand of $V_i \otimes V$, which is isomorphic to $V_{i+1}$. 
\end{proposition}
\begin{proof}
We have $V^{\otimes n} = V^{\otimes (n-1)} \otimes V = (V_{n-1} \oplus W_1 \oplus \cdots \oplus W_m)\otimes V$, where $\dim W_i$ is even for all $i$. Then, $V^{\otimes n} = (V_{n-1} \otimes V) \oplus (W_1 \otimes V) \oplus \cdots \oplus (W_m \otimes V)$. The only odd-dimensional indecomposable summand of this is $V_n$, by assumption. Since $\dim (V_{n-1}\otimes V)$ is odd, then $V_n$ must be an indecomposable summand of $V_{n-1}\otimes V$.
\end{proof}

For self-dual modules, it is straightforward to compute $P_V(n)$, as follows:

\begin{proposition}\label{prop:symm180tenspow}
    If the monomial diagram of $V$ is symmetric by rotation of $180^\circ$, then $V_{\normalfont\text{odd}} = V$ and $V_{\normalfont\text{even}} = \kk$. In particular,
    $$P_V(n) = \begin{cases} 
      \dim V & \text{if } n \text{ odd} \\
      1 & \text{if } n \text{ even}.
   \end{cases}$$
\end{proposition}
\begin{proof}
By Lemma~\ref{lem:monomialdual180rotate}, $V$ is self-dual. Recall that the trivial representation $\kk$ is a summand of \\ $V \otimes V^* \cong V \otimes V$, so $P_V(2) = 1$. Then $V_2 \otimes V = \kk \otimes V = V$, so $V_3 = V$. Thus $P_V(3) = \dim V$. By induction, we find that $V_{\text{odd}} = V$ and $V_{\text{even}} = \kk$.
\end{proof}

\subsection{``Staircase" monomial representations}\label{subsec:tenspowstaircase}

We next consider the monomial representations for $\alpha(1,1)$. We refer to these representations in the following way, based on their diagrams' appearance.

\begin{definition}
    Call the monomial representation given by the partitions 
    \[
    (m,m-1,\dots,1)/(m-2,m-3,\dots,1)
    \]
    the \textit{$m$-staircase} monomial representation.
\end{definition}

From the classification of all indecomposable summands of $\mathbb{Z}/2\mathbb{Z} \times \mathbb{Z}/2\mathbb{Z}$ in~\cite[pp. 176]{bensonz2xz2}, we know that $\Omega^m(\kk)$ are these odd-dimensional indecomposable representations, where $\Omega^m(V)$ is the $m$-th syzygy, defined as the kernel of the $m$-th map in the projective resolution of $V$. We can calculate these syzygies and it can be shown that the $m$-staircase monomial representation is $\Omega^{1-m}(\kk)$. To examine tensor powers, we need two lemmas. The following lemma is found in~\cite[Corollary~3.1.6]{bensonsyzygy} in the finite group case, and the same proof holds for any finite group scheme, in particular $\alpha(r,s)$:

\begin{lemma}\label{lem:syzygytensor}
    If $V$ and $W$ are representations of $\kk G$ for a finite group $G$, then $\Omega(V) \otimes W$ is isomorphic to $\Omega(V \otimes W)$, up to projective summands.
\end{lemma}

Also, we need the following lemma given in~\cite[Corollary 8.1.3]{webbrep}:

\begin{lemma}\label{lem:projReven}
    If $\kk$ is an algebraically closed field of characteristic $2$ and $2$ divides $|G|$, then all projective representations of $G$ have even dimension.
\end{lemma}

We can prove the following proposition about the tensor powers of staircase monomial representations:

\begin{proposition}
    If $V$ is the $m$-staircase representation, then $V_n$ is the $(mn-n+1)$-staircase monomial representation.
\end{proposition}
\begin{proof}
We have that $V = \Omega^{1-m}(\kk)$. Then $V^{\otimes n} \cong (\Omega^{1-m}(\kk))^{\otimes n}$. By Lemma~\ref{lem:syzygytensor}, this is isomorphic, up to projective summands, to $\Omega^{(1-m)n}(\kk^{\otimes n}) = \Omega^{n-mn}(\kk)$. By Lemma~\ref{lem:projReven}, all projective representations are even dimensional, so $V_n$, the odd-dimensional summand of $V^{\otimes n}$, is also the odd-dimensional summand of $\Omega^{n-mn}(\kk)$. However, $\Omega^{n-mn}(\kk)$ is the $(mn-n+1)$-staircase monomial representation, which has odd-dimension and is indecomposable. Thus $V_n$ is the $(mn-n+1)$-staircase monomial representation.
\end{proof}

It immediately follows that Conjecture~\ref{conj:tensorpowers} is satisfied in this case.

\begin{corollary}
    If $V$ is the $m$-staircase module, then the function $P_V(n)$ is given by the linear polynomial $(2m-2)n+1$.
\end{corollary}

\begin{remark}
    It is shown in~\cite{benson} that Conjecture~\ref{conj:p2div4} is true for $G=\mathbb{Z}/2\mathbb{Z} \times \mathbb{Z}/2\mathbb{Z}$ in general.
\end{remark}

\subsection{$(4,1)$ Monomial representation}\label{subsec:graded41rep}

Let $V$ be the monomial representation of $\mathbb{Z}/2\mathbb{Z} \times \mathbb{Z}/4\mathbb{Z}$ (or $\alpha(1,2)$) corresponding to the partition $(4,1)$.

\begin{proposition}\label{prop:41decomp}
We have the following decomposition into indecomposable summands:
$$V_{2k} \otimes V = V_{2k+1} \oplus \underbrace{F \oplus \cdots \oplus F}_{4k \text{\normalfont\ copies}} \quad \text{ and } \quad V_{2k-1} \otimes V = V_{2k} \oplus W \oplus W \oplus \underbrace{F \oplus \cdots \oplus F}_{4k-3 \text{\normalfont\ copies}},$$
where $F$ is a free module of dimension $8$ and $W$ is dimension $4$. 
\end{proposition}

In order to prove this, we explicitly give the decomposition, via the following stronger lemma.

\begin{lemma}\label{lem:41partition}
We claim the following.
\begin{enumerate}[leftmargin=*]
    \item The representation $V_{2k}$ is the direct sum of $1$-dimensional homogeneous components $V_{4k,2k}$ and $V_{2k-1+i,j}$ where $j=6k-1-2i, \dots, 6k+2-2i$ for $i=1,3,5, \dots, 2k-1$ and\\ $j=6k-2i, \dots, 6k+3-2i$ for $i=2,4,6,\dots,2k$.
    
    \item The representation $V_{2k-1}$ is the direct sum of $1$-dimensional homogeneous components $V_{4k-2,2k-1}$ and $V_{2k-2+i,j}$ where $j=6k-3-2i, \dots, 6k-2i$ for $i=1,3,5,\dots, 2k-1$ and\\ $j=6k-4-2i,\dots,6k-1-2i$ for $i=2,4,6,\dots,2k-2$.
    
    \item The representation $V_{2k} \otimes V$ decomposes into the summands whose graded monomial diagrams are shown in Table~\ref{tab:41part-lem-3}.

    \item The representation $V_{2k-1} \otimes V$ decomposes into the summands whose graded monomial diagrams are shown in Table~\ref{tab:41part-lem-4}.
\end{enumerate}
\end{lemma}

\begin{table}[H]
\centering
\caption{Indecomposable summands of $V_{2k} \otimes V$.}
\label{tab:41part-lem-3}
\scalebox{0.75}{
\makebox[\textwidth][c]{
\begin{tabular}{cccccc}
    \hline
    \multicolumn{3}{|c|}{$V_{2k+1}$ summand:} & \multicolumn{3}{|c|}{Dimension $8$ Family 1 summand:}
    \\\hline
    \multicolumn{3}{|c|}{
    \begin{tikzpicture}
        \draw (0,0) rectangle node{$v_{2k,6k} \otimes v_{1,1}$} ++(2.2,0.85);
        \draw (0,0.85) rectangle node{$v_{2k,6k} \otimes v_{1,2}$} ++(2.2,0.85);
        \draw (0,1.7) rectangle node{$v_{2k,6k} \otimes v_{1,3}$} ++(2.2,0.85);
        \draw (0,2.55) rectangle node{$v_{2k,6k} \otimes v_{1,4}$} ++(2.2,0.85);
        \draw (2.2,0.85) -- (3.2,0.85) -- (3.2,0) -- (4.2,0) -- (4.2,-2.5) -- (5.2,-2.5) -- (5.2,-3.35) -- (3.2,-3.35) -- (3.2,-2.5) -- (2.2,-2.5) -- (2.2,0);
        \node at (3.2,-1.25) {$V_{2k} \otimes v_{2,1}$};
    \end{tikzpicture}
    }
    &
    \multicolumn{3}{|c|}{
    \begin{tikzpicture}
        \draw (0,0) rectangle node{$v_{2k-1+i,6k-2i} \otimes v_{1,1}$} ++(4.15,1.25);
        \draw (4.15,0) rectangle node{$\begin{aligned}&v_{2k+i,6k-2i} \otimes v_{1,1}\\&+ v_{2k-1+i,6k-2i} \otimes v_{2,1}\end{aligned}$} ++(4,1.25);
        \draw (0,1.25) rectangle node{$\begin{aligned}&v_{2k-1+i,6k+1-2i}\otimes v_{1,1}\\&+ v_{2k-1+i,6k-2i}\otimes v_{1,2}\end{aligned}$} ++(4.15,1.25);
        \draw (4.15,1.25) rectangle node{$\begin{aligned}&v_{2k-1+i,6k+1-2i} \otimes v_{2,1}\\&+v_{2k+i,6k-2i}\otimes v_{1,2}\end{aligned}$} ++(4,1.25);
        \draw (0,2.5) rectangle node{$\begin{aligned}&v_{2k-1+i,6k+2-2i} \otimes v_{1,1}\\&+v_{2k-1+i,6k-2i}\otimes v_{1,3}\end{aligned}$} ++(4.15,1.25);
        \draw (4.15,2.5) rectangle node{$\begin{aligned}&v_{2k-1+i,6k+2-2i} \otimes v_{2,1}\\&+v_{2k+i,6k-2i}\otimes v_{1,3}\end{aligned}$} ++(4,1.25);
        \draw (0,3.75) rectangle node{$\begin{aligned}&v_{2k-1+i,6k+3-2i} \otimes v_{1,1}\\&+v_{2k-1+i,6k+2-2i}\otimes v_{1,2}\\&+v_{2k-1+i,6k+1-2i}\otimes v_{1,3}\\&+v_{2k-1+i,6k-2i}\otimes v_{1,4}\end{aligned}$} ++(4.15,2.5);
        \draw (4.15,3.75) rectangle node{$\begin{aligned}&v_{2k-1+i,6k+3-2i} \otimes v_{2,1}\\&+v_{2k+i,6k-2i}\otimes v_{1,4}\end{aligned}$} ++(4,2.5);
    \end{tikzpicture}
    }
    \\
    \multicolumn{3}{|c|}{} & \multicolumn{3}{|c|}{for $i=2,4,6,\dots,2k$}
    \\\hline
    \multicolumn{2}{|c|}{Dimension $8$ Family 2 summand:} & \multicolumn{2}{|c|}{Dimension $8$ Family 3 summand:} & \multicolumn{2}{|c|}{Dimension $8$ Family 4 summand:}
    \\\hline
    \multicolumn{2}{|c|}{
    \begin{tikzpicture}
        \draw (0,0) rectangle node{$v_{2k-1+i,6k-1-2i} \otimes v_{1,1}$} ++(4.15,1.2);
        \draw (4.15,0) rectangle node{$\begin{aligned}&v_{2k+i,6k-1-2i} \otimes v_{1,1}\\&+ v_{2k-1+i,6k-1-2i} \otimes v_{2,1}\end{aligned}$} ++(4.15,1.2);
        \draw (0,1.2) rectangle node{$\begin{aligned}&v_{2k-1+i,6k-2i}\otimes v_{1,1}\\&+ v_{2k-1+i,6k-1-2i}\otimes v_{1,2}\end{aligned}$} ++(4.15,1.5);
        \draw (4.15,1.2) rectangle node{$\begin{aligned}&v_{2k+i,6k-1-2i} \otimes v_{1,2}\\&+v_{2k+i,6k-2i}\otimes v_{1,1}\\&+v_{2k-1+i,6k-2i}\otimes v_{2,1}\end{aligned}$} ++(4.15,1.5);
        \draw (0,2.7) rectangle node{$\begin{aligned}&v_{2k-1+i,6k+1-2i} \otimes v_{1,1}\\&+v_{2k-1+i,6k-1-2i}\otimes v_{1,3}\end{aligned}$} ++(4.15,1.5);
        \draw (4.15,2.7) rectangle node{$\begin{aligned}&v_{2k+i,6k+1-2i} \otimes v_{1,1}\\&+v_{2k+i,6k-1-2i}\otimes v_{1,3}\\&+v_{2k-1+i,6k+1-2i}\otimes v_{2,1}\end{aligned}$} ++(4.15,1.5);
        \draw (0,4.2) rectangle node{$\begin{aligned}&v_{2k-1+i,6k+2-2i} \otimes v_{1,1}\\&+v_{2k-1+i,6k+1-2i}\otimes v_{1,2}\\&+v_{2k-1+i,6k-2i}\otimes v_{1,3}\\&+v_{2k-1+i,6k-1-2i}\otimes v_{1,4}\end{aligned}$} ++(4.15,2);
        \draw (4.15,4.2) rectangle node{$\begin{aligned}&v_{2k-1+i,6k+2-2i} \otimes v_{2,1}\\&+v_{2k+i,6k+1-2i}\otimes v_{1,2}\\&+v_{2k+i,6k-2i}\otimes v_{1,3}\\&+v_{2k+i,6k-1-2i}\otimes v_{1,4}\end{aligned}$} ++(4.15,2);
    \end{tikzpicture}
    }
    &
    \multicolumn{2}{|c|}{
    \begin{tikzpicture}
        \draw (0,0) rectangle node{$v_{2k-1+i,6k-2i} \otimes v_{1,1}$} ++(4.15,1.2);
        \draw (4.15,0) rectangle node{$\begin{aligned}&v_{2k+i,6k-2i} \otimes v_{1,1}\\&+ v_{2k-1+i,6k-2i} \otimes v_{2,1}\end{aligned}$} ++(4.15,1.2);
        \draw (0,1.2) rectangle node{$\begin{aligned}&v_{2k-1+i,6k+1-2i}\otimes v_{1,1}\\&+ v_{2k-1+i,6k-2i}\otimes v_{1,2}\end{aligned}$} ++(4.15,1.5);
        \draw (4.15,1.2) rectangle node{$\begin{aligned}&v_{2k+i,6k+1-2i} \otimes v_{1,1}\\&+v_{2k+i,6k-2i}\otimes v_{1,2}\\&+v_{2k-1+i,6k+1-2i}\otimes v_{2,1}\end{aligned}$} ++(4.15,1.5);
        \draw (0,2.7) rectangle node{$\begin{aligned}&v_{2k-1+i,6k+2-2i} \otimes v_{1,1}\\&+v_{2k-1+i,6k-2i}\otimes v_{1,3}\end{aligned}$} ++(4.15,1.25);
        \draw (4.15,2.7) rectangle node{$\begin{aligned}&v_{2k+i,6k-2i} \otimes v_{1,3}\\&+v_{2k-1+i,6k+2-2i}\otimes v_{2,1}\end{aligned}$} ++(4.15,1.25);
        \draw (0,3.95) rectangle node{$\begin{aligned}&v_{2k-1+i,6k+2-2i} \otimes v_{1,2}\\&+v_{2k-1+i,6k+1-2i}\otimes v_{1,3}\\&+v_{2k-1+i,6k-2i}\otimes v_{1,4}\end{aligned}$} ++(4.15,2);
        \draw (4.15,3.95) rectangle node{$\begin{aligned}&v_{2k+i,6k+1-2i} \otimes v_{1,3}\\&+v_{2k+i,6k-2i}\otimes v_{1,4}\end{aligned}$} ++(4.15,2);
    \end{tikzpicture}
    }
    &
    \multicolumn{2}{|c|}{
    \begin{tikzpicture}
        \draw (0,0) rectangle node{$v_{2k-1+i,6k+1-2i} \otimes v_{1,1}$} ++(4.15,1.2);
        \draw (4.15,0) rectangle node{$\begin{aligned}&v_{2k+i,6k+1-2i} \otimes v_{1,1}\\&+ v_{2k-1+i,6k+1-2i} \otimes v_{2,1}\end{aligned}$} ++(4.15,1.2);
        \draw (0,1.2) rectangle node{$\begin{aligned}&v_{2k-1+i,6k+2-2i}\otimes v_{1,1}\\&+ v_{2k-1+i,6k+1-2i}\otimes v_{1,2}\end{aligned}$} ++(4.15,1.5);
        \draw (4.15,1.2) rectangle node{$\begin{aligned}&v_{2k+i,6k+1-2i} \otimes v_{1,2}\\&+v_{2k-1+i,6k+2-2i}\otimes v_{2,1}\end{aligned}$} ++(4.15,1.5);
        \draw (0,2.7) rectangle node{$v_{2k-1+i,6k+1-2i} \otimes v_{1,3}$} ++(4.15,1.25);
        \draw (4.15,2.7) rectangle node{$v_{2k+i,6k+1-2i} \otimes v_{1,3}$} ++(4.15,1.25);
        \draw (0,3.95) rectangle node{$\begin{aligned}&v_{2k-1+i,6k+2-2i} \otimes v_{1,3}\\&+v_{2k-1+i,6k+1-2i}\otimes v_{1,4}\end{aligned}$} ++(4.15,1.5);
        \draw (4.15,3.95) rectangle node{$v_{2k+i,6k+1-2i} \otimes v_{1,4}$} ++(4.15,1.5);
    \end{tikzpicture}
    }
    \\
    \multicolumn{2}{|c|}{for $i=1,3,5,\dots, 2k-1$} & \multicolumn{2}{|c|}{for $i=1,3,5,\dots, 2k-1$} & \multicolumn{2}{|c|}{for $i=1,3,5,\dots, 2k-1$}
    \\\hline
\end{tabular}
}
}
\end{table}

\begin{table}[H]
    \centering
    \caption{Indecomposable summands of $V_{2k-1}\otimes V$.}
    \label{tab:41part-lem-4}
    \scalebox{0.75}{
    \makebox[\textwidth][c]{
    \begin{tabular}{|c|cc|c|}
    \hline
    $V_{2k}$ summand: & $W_1$ summand: & \multicolumn{1}{|c|}{$W_2$ summand:} & Dimension $8$ Family 1 summand:
    \\\hline
    \begin{tikzpicture}
        \draw (0,0) rectangle node{$v_{2k-1,6k-4} \otimes v_{1,1}$} ++(3.5,0.75);
        \draw (0,0.75) rectangle node{$\begin{aligned}&v_{2k-1,6k-4} \otimes v_{1,2}\\&+v_{2k-1,6k-3}\otimes v_{1,1}\end{aligned}$} ++(3.5,1.25);
        \draw (0,2) rectangle node{$\begin{aligned}&v_{2k-1,6k-4}\otimes v_{1,3}\\&+v_{2k-1,6k-2}\otimes v_{1,1}\end{aligned}$} ++(3.5,1.25);
        \draw (0,3.25) rectangle node{$\begin{aligned}&v_{2k-1,6k-2}\otimes v_{1,2}\\&+v_{2k-1,6k-3}\otimes v_{1,3}\\&+v_{2k-1,6k-4}\otimes v_{1,4}\end{aligned}$} ++(3.5,1.75);
        \draw (3.5,3.25) -- (4.5,3.25) -- (4.5,0) -- (5.5,0) -- (5.5,-0.75) -- (6.5,-0.75) -- (6.5,-3) -- (7.5,-3) -- (7.5,-3.75) -- (5.5,-3.75) -- (5.5,-3) -- (4.5,-3) -- (4.5,-0.75) -- (3.5,-0.75) -- (3.5,0);
        \node at (5.5,-1.75) {$V_{2k-1} \otimes v_{2,1}$};
    \end{tikzpicture}
    &
    \begin{tikzpicture}
        \draw (0,0) rectangle node{$v_{2k-1,6k-4} \otimes v_{1,2}$} ++(3.5,1.5);
        \draw (0,1.5) rectangle node{$\begin{aligned}&v_{2k-1,6k-3}\otimes v_{1,2}\\&+v_{2k-1,6k-4}\otimes v_{1,3}\end{aligned}$} ++(3.5,1.5);
        \draw (0,3) rectangle node{$\begin{aligned}&v_{2k-1,6k-2}\otimes v_{1,2}\\&+v_{2k-1,6k-4}\otimes v_{1,4}\end{aligned}$} ++(3.5,1.5);
        \draw (0,4.5) rectangle node{$\begin{aligned}&v_{2k-1,6k-2}\otimes v_{1,3}\\&+v_{2k-1,6k-3}\otimes v_{1,4}\end{aligned}$} ++(3.5,1.5);
    \end{tikzpicture}
    &
    \multicolumn{1}{|c|}{
    \begin{tikzpicture}
        \draw (0,0) rectangle node{$v_{2k-1,6k-3}\otimes v_{1,2}$} ++(3.5,1.5);
        \draw (0,1.5) rectangle node{$\begin{aligned}&v_{2k-1,6k-2}\otimes v_{1,2}\\&+v_{2k-1,6k-3}\otimes v_{1,3}\end{aligned}$} ++(3.5,1.5);
        \draw (0,3) rectangle node{$v_{2k-1,6k-3}\otimes v_{1,4}$} ++(3.5,1.5);
        \draw (0,4.5) rectangle node{$v_{2k-1,6k-2} \otimes v_{1,4}$} ++(3.5,1.5);
    \end{tikzpicture}
    }
    &
    \begin{tikzpicture}
        \draw (0,0) rectangle node{$v_{2k-2+i,6k-3-2i} \otimes v_{1,1}$} ++(4.15,1.5);
        \draw (4.15,0) rectangle node{$\begin{aligned}&v_{2k-1+i,6k-3-2i} \otimes v_{1,1}\\&+ v_{2k-2+i,6k-3-2i} \otimes v_{2,1}\end{aligned}$} ++(4.15,1.5);
        \draw (0,1.5) rectangle node{$\begin{aligned}&v_{2k-2+i,6k-2-2i}\otimes v_{1,1}\\&+ v_{2k-2+i,6k-3-2i}\otimes v_{1,2}\end{aligned}$} ++(4.15,1.5);
        \draw (4.15,1.5) rectangle node{$\begin{aligned}&v_{2k-2+i,6k-2-2i} \otimes v_{2,1}\\&+v_{2k-1+i,6k-3-2i}\otimes v_{1,2}\end{aligned}$} ++(4.15,1.5);
        \draw (0,3) rectangle node{$\begin{aligned}&v_{2k-2+i,6k-1-2i} \otimes v_{1,1}\\&+v_{2k-2+i,6k-3-2i}\otimes v_{1,3}\end{aligned}$} ++(4.15,1.5);
        \draw (4.15,3) rectangle node{$\begin{aligned}&v_{2k-2+i,6k-1-2i} \otimes v_{2,1}\\&+v_{2k-1+i,6k-3-2i}\otimes v_{1,3}\end{aligned}$} ++(4.15,1.5);
        \draw (0,4.5) rectangle node{$\begin{aligned}&v_{2k-2+i,6k-2i} \otimes v_{1,1}\\&+v_{2k-2+i,6k-1-2i}\otimes v_{1,2}\\&+v_{2k-2+i,6k-2-2i}\otimes v_{1,3}\\&+v_{2k-2+i,6k-3-2i}\otimes v_{1,4}\end{aligned}$} ++(4.15,3);
        \draw (4.15,4.5) rectangle node{$\begin{aligned}&v_{2k-2+i,6k-2i} \otimes v_{2,1}\\&+v_{2k-1+i,6k-3-2i}\otimes v_{1,4}\end{aligned}$} ++(4.15,3);
    \end{tikzpicture}
    \\
    & & \multicolumn{1}{|c|}{} & for $i=1,3,5,\dots,2k-1$
    \\\hline
    Dimension $8$ Family 2 summand: & \multicolumn{2}{|c|}{Dimension $8$ Family 3 summand:} & Dimension $8$ Family 4 summand:
    \\\hline
    \begin{tikzpicture}
        \draw (0,0) rectangle node{$v_{2k-2+i,6k-4-2i} \otimes v_{1,1}$} ++(4.15,1.25);
        \draw (4.15,0) rectangle node{$\begin{aligned}&v_{2k-1+i,6k-4-2i} \otimes v_{1,1}\\&+ v_{2k-2+i,6k-4-2i} \otimes v_{2,1}\end{aligned}$} ++(4.15,1.25);
        \draw (0,1.25) rectangle node{$\begin{aligned}&v_{2k-2+i,6k-3-2i}\otimes v_{1,1}\\&+ v_{2k-2+i,6k-4-2i}\otimes v_{1,2}\end{aligned}$} ++(4.15,2);
        \draw (4.15,1.25) rectangle node{$\begin{aligned}&v_{2k-1+i,6k-3-2i} \otimes v_{1,1}\\&+v_{2k-1+i,6k-4-2i}\otimes v_{1,2}\\&+v_{2k-2+i,6k-3-2i}\otimes v_{2,1}\end{aligned}$} ++(4.15,2);
        \draw (0,3.25) rectangle node{$\begin{aligned}&v_{2k-2+i,6k-2-2i} \otimes v_{1,1}\\&+v_{2k-2+i,6k-4-2i}\otimes v_{1,3}\end{aligned}$} ++(4.15,2);
        \draw (4.15,3.25) rectangle node{$\begin{aligned}&v_{2k-1+i,6k-2-2i} \otimes v_{1,1}\\&+v_{2k-1+i,6k-4-2i}\otimes v_{1,3}\\&+v_{2k-2+i,6k-2-2i}\otimes v_{2,1}\end{aligned}$} ++(4.15,2);
        \draw (0,5.25) rectangle node{$\begin{aligned}&v_{2k-2+i,6k-1-2i} \otimes v_{1,1}\\&+v_{2k-2+i,6k-2-2i}\otimes v_{1,2}\\&+v_{2k-2+i,6k-3-2i}\otimes v_{1,3}\\&+v_{2k-2+i,6k-4-2i}\otimes v_{1,4}\end{aligned}$} ++(4.15,2.65);
        \draw (4.15,5.25) rectangle node{$\begin{aligned}&v_{2k-2+i,6k-1-2i} \otimes v_{2,1}\\&+v_{2k-1+i,6k-2-2i}\otimes v_{1,2}\\&+v_{2k-1+i,6k-3-2i}\otimes v_{1,3}\\&+v_{2k-1+i,6k-4-2i}\otimes v_{1,4}\end{aligned}$} ++(4.15,2.65);
    \end{tikzpicture}
    &
    \multicolumn{2}{|c|}{
    \begin{tikzpicture}
        \draw (0,0) rectangle node{$v_{2k-2+i,6k-3-2i} \otimes v_{1,1}$} ++(4.15,1.5);
        \draw (4.15,0) rectangle node{$\begin{aligned}&v_{2k-1+i,6k-3-2i} \otimes v_{1,1}\\&+ v_{2k-2+i,6k-3-2i} \otimes v_{2,1}\end{aligned}$} ++(4.15,1.5);
        \draw (0,1.5) rectangle node{$\begin{aligned}&v_{2k-2+i,6k-2-2i}\otimes v_{1,1}\\&+ v_{2k-2+i,6k-3-2i}\otimes v_{1,2}\end{aligned}$} ++(4.15,2.25);
        \draw (4.15,1.5) rectangle node{$\begin{aligned}&v_{2k-1+i,6k-2-2i} \otimes v_{1,1}\\&+v_{2k-1+i,6k-3-2i}\otimes v_{1,2}\\&+v_{2k-2+i,6k-2-2i}\otimes v_{2,1}\end{aligned}$} ++(4.15,2.25);
        \draw (0,3.75) rectangle node{$\begin{aligned}&v_{2k-2+i,6k-1-2i} \otimes v_{1,1}\\&+v_{2k-2+i,6k-3-2i}\otimes v_{1,3}\end{aligned}$} ++(4.15,1.5);
        \draw (4.15,3.75) rectangle node{$\begin{aligned}&v_{2k-1+i,6k-3-2i} \otimes v_{1,3}\\&+v_{2k-2+i,6k-1-2i}\otimes v_{2,1}\end{aligned}$} ++(4.15,1.5);
        \draw (0,5.25) rectangle node{$\begin{aligned}&v_{2k-2+i,6k-1-2i} \otimes v_{1,2}\\&+v_{2k-2+i,6k-2-2i}\otimes v_{1,3}\\&+v_{2k-2+i,6k-3-2i}\otimes v_{1,4}\end{aligned}$} ++(4.15,2.25);
        \draw (4.15,5.25) rectangle node{$\begin{aligned}&v_{2k-1+i,6k-2-2i} \otimes v_{1,3}\\&+v_{2k-1+i,6k-3-2i}\otimes v_{1,4}\end{aligned}$} ++(4.15,2.25);
    \end{tikzpicture}
    }
    &
    \begin{tikzpicture}
        \draw (0,0) rectangle node{$v_{2k-2+i,6k-2-2i} \otimes v_{1,1}$} ++(4.15,1.5);
        \draw (4.15,0) rectangle node{$\begin{aligned}&v_{2k-1+i,6k-2-2i} \otimes v_{1,1}\\&+ v_{2k-2+i,6k-2-2i} \otimes v_{2,1}\end{aligned}$} ++(4.15,1.5);
        \draw (0,1.5) rectangle node{$\begin{aligned}&v_{2k-2+i,6k-1-2i}\otimes v_{1,1}\\&+ v_{2k-2+i,6k-2-2i}\otimes v_{1,2}\end{aligned}$} ++(4.15,1.5);
        \draw (4.15,1.5) rectangle node{$\begin{aligned}&v_{2k-1+i,6k-2-2i} \otimes v_{1,2}\\&+v_{2k-2+i,6k-1-2i}\otimes v_{2,1}\end{aligned}$} ++(4.15,1.5);
        \draw (0,3) rectangle node{$v_{2k-2+i,6k-2-2i} \otimes v_{1,3}$} ++(4.15,1.5);
        \draw (4.15,3) rectangle node{$v_{2k-1+i,6k-2-2i} \otimes v_{1,3}$} ++(4.15,1.5);
        \draw (0,4.5) rectangle node{$\begin{aligned}&v_{2k-2+i,6k-1-2i} \otimes v_{1,3}\\&+v_{2k-2+i,6k-2-2i}\otimes v_{1,4}\end{aligned}$} ++(4.15,1.5);
        \draw (4.15,4.5) rectangle node{$v_{2k-1+i,6k-2-2i} \otimes v_{1,4}$} ++(4.15,1.5);
    \end{tikzpicture}
    \\
    for $i=2,4,6,\dots,2k-2$ & \multicolumn{2}{|c|}{for $i=2,4,6,\dots,2k-2$} & for $i=2,4,6,\dots,2k-2$
    \\\hline
    \end{tabular}
    }
    }
\end{table}

\begin{proof}
We proceed by induction. It can be checked that (2) is true for $k=1$. 

For the inductive step, assume (2) is true. Consider the decomposition given in (4). First, we show that each of these summands are subrepresentations. It can be checked in each case that the vectors chosen in the diagrams above are chosen such that the action of $x$ takes the displayed vector to the vector in the box adjacent to the right (or $0$ if that box does not exist) and the action of $y$ takes the displayed vector to the vector in the box adjacent above (or $0$ if that box does not exist). Since each vector is a basis for the homogeneous component it is in, then each diagram is closed under the action of $x$ and $y$, and all summands are subrepresentations.

Each of the claimed summands is indecomposable, since all of the diagrams are connected.

Finally, we must show that the direct sum of the indecomposable subrepresentations shown is the original representation $V_{2k-1}\otimes V$. First, we must know where the summands are located relative to each other. Figure~\ref{fig:41par-location} shows the dimensions of the homogeneous components of $V$, $V_1\otimes V$, $V_2\otimes V$, $V_3\otimes V$, and $V_4 \otimes V$, respectively from left to right. The black outline is the location of $V_i$, the diamonds show the positions of the bottom left cells of the dimension $8$ free modules, and the circles show the positions of the bottom cells of the dimension $4$ summands, as shown in the key. This pattern is generalized to higher $n$.

\begin{figure}[h]
     \centering
     \caption{Positions of the indecomposable summands in the first five tensor powers.}
     \label{fig:41par-location}
     \begin{subfigure}[H]{0.18\textwidth}
         \centering
         \begin{tikzpicture}[scale=0.45]
            \node at (3,0) {$1$};
            \node at (2,0) {$1$};
            \node at (2,1) {$1$};
            \node at (2,2) {$1$};
            \node at (2,3) {$1$};

            \draw[thick] (1.5,-0.5) -- (3.5,-0.5) -- (3.5,0.5) -- (2.5,0.5) -- (2.5,3.5) -- (1.5,3.5) -- (1.5,-0.5);
         \end{tikzpicture}
     \end{subfigure}
     \begin{subfigure}[H]{0.18\textwidth}
         \centering
         \begin{tikzpicture}[scale=0.45]
            \draw (1,2) circle (0.5);
            \draw[fill=lightgray] (1,3) circle (0.5);
         
            \node at (3,0) {$1$};
            \node at (2,0) {$2$};
            \node at (2,1) {$2$};
            \node at (2,2) {$2$};
            \node at (2,3) {$2$};
            
            \node at (1,0) {$1$};
            \node at (1,1) {$2$};
            \node at (1,2) {$3$};
            \node at (1,3) {$4$};
            \node at (1,4) {$3$};
            \node at (1,5) {$2$};
            \node at (1,6) {$1$};

            \draw[thick] (1.5,-0.5) -- (3.5,-0.5) -- (3.5,0.5) -- (2.5,0.5) -- (2.5,3.5) -- (1.5,3.5) -- (1.5,4.5) -- (0.5,4.5) -- (0.5,0.5) -- (1.5,0.5) -- (1.5,-0.5);
            
            \node[draw, shape=diamond, semithick] at (1,0) {};
         \end{tikzpicture}
     \end{subfigure}
     \begin{subfigure}[H]{0.18\textwidth}
         \centering
         \begin{tikzpicture}[scale=0.45]
            \node[draw, shape=diamond, semithick] at (2,0) {};
            \node[draw, fill=lightgray, shape=diamond, semithick] at (1,1) {};
            \node[draw, regular polygon, regular polygon sides=3, scale=0.85, semithick] at (1,2-0.1) {};
            \node[draw, fill=lightgray, regular polygon, regular polygon sides=3, scale=0.85, semithick] at (1,3-0.1) {};
            
            \node at (4,0) {$1$};
            \node at (3,0) {$2$};
            \node at (3,1) {$2$};
            \node at (3,2) {$2$};
            \node at (3,3) {$2$};
        
            \node at (2,0) {$1$};
            \node at (2,1) {$3$};
            \node at (2,2) {$4$};
            \node at (2,3) {$5$};
            \node at (2,4) {$4$};
            \node at (2,5) {$2$};
            \node at (2,6) {$1$};
        
            \node at (1,1) {$1$};
            \node at (1,2) {$2$};
            \node at (1,3) {$3$};
            \node at (1,4) {$4$};
            \node at (1,5) {$3$};
            \node at (1,6) {$2$};
            \node at (1,7) {$1$};
        
            \draw[thick] (2.5,-0.5) -- (4.5,-0.5) -- (4.5,0.5) -- (3.5,0.5) -- (3.5,3.5) -- (2.5,3.5) -- (2.5,4.5) -- (1.5,4.5) -- (1.5,7.5) -- (0.5,7.5) -- (0.5,3.5) -- (1.5,3.5) -- (1.5,0.5) -- (2.5,0.5) -- (2.5,-0.5);
        \end{tikzpicture}
     \end{subfigure}
     \begin{subfigure}[H]{0.18\textwidth}
         \centering
         \begin{tikzpicture}[scale=0.45]
            \draw (0,6) circle (0.5);
            \draw[fill=lightgray] (0,7) circle (0.5);

            \node[draw, shape=diamond, semithick] at (2,0) {};
            \node[draw, fill=lightgray, shape=diamond, semithick] at (1,1) {};
            \node[draw, regular polygon, regular polygon sides=3, scale=0.85, semithick] at (1,2-0.1) {};
            \node[draw, fill=lightgray, regular polygon, regular polygon sides=3, scale=0.85, semithick] at (1,3-0.1) {};
            \node[draw, shape=diamond, semithick] at (0,4) {};
         
            \node at (4,0) {$1$};
            \node at (3,0) {$2$};
            \node at (3,1) {$2$};
            \node at (3,2) {$2$};
            \node at (3,3) {$2$};
        
            \node at (2,0) {$1$};
            \node at (2,1) {$3$};
            \node at (2,2) {$4$};
            \node at (2,3) {$5$};
            \node at (2,4) {$4$};
            \node at (2,5) {$2$};
            \node at (2,6) {$1$};
        
            \node at (1,1) {$1$};
            \node at (1,2) {$2$};
            \node at (1,3) {$3$};
            \node at (1,4) {$5$};
            \node at (1,5) {$4$};
            \node at (1,6) {$3$};
            \node at (1,7) {$2$};
        
            \node at (0,4) {$1$};
            \node at (0,5) {$2$};
            \node at (0,6) {$3$};
            \node at (0,7) {$4$};
            \node at (0,8) {$3$};
            \node at (0,9) {$2$};
            \node at (0,10) {$1$};
        
            \draw[thick] (2.5,-0.5) -- (4.5,-0.5) -- (4.5,0.5) -- (3.5,0.5) -- (3.5,3.5) -- (2.5,3.5) -- (2.5,4.5) -- (1.5,4.5) -- (1.5,7.5) -- (0.5,7.5) -- (0.5,8.5) -- (-0.5,8.5) -- (-0.5,4.5) -- (0.5,4.5) -- (0.5,3.5) -- (1.5,3.5) -- (1.5,0.5) -- (2.5,0.5) -- (2.5,-0.5);
        \end{tikzpicture}
     \end{subfigure}
     \begin{subfigure}[H]{0.18\textwidth}
         \centering
         \begin{tikzpicture}[scale=0.45]
            \node[draw, shape=diamond, semithick] at (2,0) {};
            \node[draw, fill=lightgray, shape=diamond, semithick] at (1,1) {};
            \node[draw, regular polygon, regular polygon sides=3, scale=0.85, semithick] at (1,2-0.1) {};
            \node[draw, fill=lightgray, regular polygon, regular polygon sides=3, scale=0.85, semithick] at (1,3-0.1) {};
            \node[draw, shape=diamond, semithick] at (0,4) {};
            \node[draw, fill=lightgray, shape=diamond, semithick] at (-1,5) {};
            \node[draw, regular polygon, regular polygon sides=3, scale=0.85, semithick] at (-1,6-0.1) {};
            \node[draw, fill=lightgray, regular polygon, regular polygon sides=3, scale=0.85, semithick] at (-1,7-0.1) {};
            
            \node at (4,0) {$1$};
            \node at (3,0) {$2$};
            \node at (3,1) {$2$};
            \node at (3,2) {$2$};
            \node at (3,3) {$2$};
        
            \node at (2,0) {$1$};
            \node at (2,1) {$3$};
            \node at (2,2) {$4$};
            \node at (2,3) {$5$};
            \node at (2,4) {$4$};
            \node at (2,5) {$2$};
            \node at (2,6) {$1$};
        
            \node at (1,1) {$1$};
            \node at (1,2) {$2$};
            \node at (1,3) {$3$};
            \node at (1,4) {$5$};
            \node at (1,5) {$4$};
            \node at (1,6) {$3$};
            \node at (1,7) {$2$};
        
            \node at (0,4) {$1$};
            \node at (0,5) {$3$};
            \node at (0,6) {$4$};
            \node at (0,7) {$5$};
            \node at (0,8) {$4$};
            \node at (0,9) {$2$};
            \node at (0,10) {$1$};
        
            \node at (-1,5) {$1$};
            \node at (-1,6) {$2$};
            \node at (-1,7) {$3$};
            \node at (-1,8) {$4$};
            \node at (-1,9) {$3$};
            \node at (-1,10) {$2$};
            \node at (-1,11) {$1$};
        
            \draw[thick] (2.5,-0.5) -- (4.5,-0.5) -- (4.5,0.5) -- (3.5,0.5) -- (3.5,3.5) -- (2.5,3.5) -- (2.5,4.5) -- (1.5,4.5) -- (1.5,7.5) -- (0.5,7.5) -- (0.5,8.5) -- (-0.5,8.5) -- (-0.5,11.5) -- (-1.5,11.5) -- (-1.5,7.5) -- (-0.5,7.5) -- (-0.5,4.5) -- (0.5,4.5) -- (0.5,3.5) -- (1.5,3.5) -- (1.5,0.5) -- (2.5,0.5) -- (2.5,-0.5);
        \end{tikzpicture}
     \end{subfigure}
     \begin{subtable}[H]{0.6\textwidth}
         \centering
         \caption*{Key}
         \begin{tabular}{|c|c|c|c|c|c|}
            \hline
            $W_1$ & $W_2$ & Family 1 & Family 2 & Family 3 & Family 4
            \\\hline
            \begin{tikzpicture}[scale=0.4]
                \node at (0,0) {$1$};
                \node at (0,1) {$1$};
                \node at (0,2) {$1$};
                \node at (0,3) {$1$};
        
                \draw (0,0) circle (0.5);
            \end{tikzpicture}
            &
            \begin{tikzpicture}[scale=0.4]
                \draw[fill=lightgray] (0,0) circle (0.5);
            
                \node at (0,0) {$1$};
                \node at (0,1) {$1$};
                \node at (0,2) {$1$};
                \node at (0,3) {$1$};
            \end{tikzpicture}
            &
            \begin{tikzpicture}[scale=0.4]
                \node at (0,0) {$1$};
                \node at (0,1) {$1$};
                \node at (0,2) {$1$};
                \node at (0,3) {$1$};
        
                \node at (1,0) {$1$};
                \node at (1,1) {$1$};
                \node at (1,2) {$1$};
                \node at (1,3) {$1$};
        
                \node[draw, shape=diamond, semithick] at (0,0) {};
            \end{tikzpicture}
            &
            \begin{tikzpicture}[scale=0.4]
                \node[draw, fill=lightgray, shape=diamond, semithick] at (0,0) {};
            
                \node at (0,0) {$1$};
                \node at (0,1) {$1$};
                \node at (0,2) {$1$};
                \node at (0,3) {$1$};
        
                \node at (1,0) {$1$};
                \node at (1,1) {$1$};
                \node at (1,2) {$1$};
                \node at (1,3) {$1$};
            \end{tikzpicture}
            &
            \begin{tikzpicture}[scale=0.4]
                \node at (0,0) {$1$};
                \node at (0,1) {$1$};
                \node at (0,2) {$1$};
                \node at (0,3) {$1$};
        
                \node at (1,0) {$1$};
                \node at (1,1) {$1$};
                \node at (1,2) {$1$};
                \node at (1,3) {$1$};
        
                \node[draw, regular polygon, regular polygon sides=3, scale=0.85, semithick] at (0,0-0.1) {};
            \end{tikzpicture}
            &
            \begin{tikzpicture}[scale=0.4]
                \node[draw, fill=lightgray, regular polygon, regular polygon sides=3, scale=0.85, semithick] at (0,0-0.1) {};
            
                \node at (0,0) {$1$};
                \node at (0,1) {$1$};
                \node at (0,2) {$1$};
                \node at (0,3) {$1$};
        
                \node at (1,0) {$1$};
                \node at (1,1) {$1$};
                \node at (1,2) {$1$};
                \node at (1,3) {$1$};
            \end{tikzpicture}
            \\\hline    
        \end{tabular}
     \end{subtable}
\end{figure}

For every homogeneous component of $V_{2k-1}\otimes V$, we can check that the vectors in the shown summands span the vector space. For example, it can be checked that
$$\begin{aligned}
    &\{v_{2k-1,6k-2}\otimes v_{1,2}+v_{2k-1,6k-3}\otimes v_{1,3}+v_{2k-1,6k-4}\otimes v_{1,4},\\
    &\quad v_{2k-1,6k-2}\otimes v_{1,2}+v_{2k-1,6k-3}\otimes v_{1,3},\\
    &\quad v_{2k-1,6k-2}\otimes v_{1,2}+v_{2k-1,6k-4}\otimes v_{1,4}\}
\end{aligned}$$
spans the homogeneous component of degree $(2k,6k)$. We do similar verifications for the other homogeneous components, which is omitted from this proof. This means that (4) follows from (2). By a similar argument, it can be shown that (3) follows from (1). Considering the degrees of the homogeneous components show that (1) and (2) follow from (4) and (3), respectively. By induction, the lemma follows.
\end{proof}

Proposition~\ref{prop:41decomp} follows from this lemma, and as a corollary we observe that Conjecture~\ref{conj:tensorpowers} is satisfied in this case.

\begin{corollary}\label{4-1-poly}
The function $P_V(n)$ is given by $4n+1$, a linear polynomial.
\end{corollary}

\subsection{$(4,1)$ Monomial representation using syzygies}\label{subsec:gradedsyzygy}

As in the previous section, let $V$ be the monomial representation given by the partition $(4,1)$ and let $G=\mathbb{Z}/2\mathbb{Z} \times \mathbb{Z}/4\mathbb{Z}$. The monomial diagram of $V$ and the monomial diagram for the free $G$-module of rank 1 are given, respectively:

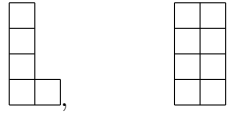
\begin{center}
\begin{tabular}{ccc}
\begin{tikzpicture}[scale=0.34]
    \draw (0,0) -- (2,0);
    \draw (0,1) -- (2,1);
    \draw (0,2) -- (1,2);
    \draw (0,3) -- (1,3);
    \draw (0,4) -- (1,4);

    \draw (0,0) -- (0,4);
    \draw (1,0) -- (1,4);
    \draw (2,0) -- (2,1);
\end{tikzpicture},   
&
$\qquad$
&
\begin{tikzpicture}[scale=0.34]
    \draw (0,0) -- (2,0);
    \draw (0,1) -- (2,1);
    \draw (0,2) -- (2,2);
    \draw (0,3) -- (2,3);
    \draw (0,4) -- (2,4);

    \draw (0,0) -- (0,4);
    \draw (1,0) -- (1,4);
    \draw (2,0) -- (2,4);
\end{tikzpicture}.
\end{tabular}
\end{center}

Since the free $G$-module of rank $1$ is indecomposable, all projective modules are free modules. The syzygy for $V$ appears as the first module in the following short exact sequence:

\begin{center}
\begin{tabular}{c}
\begin{tikzpicture}[scale=0.34]
    \node at (0,0) {$0$};
    \node at (2,0) {$\to$};
    \draw (4,-1.5) -- (5,-1.5);
    \draw (4,-0.5) -- (5,-0.5);
    \draw (4,0.5) -- (5,0.5);
    \draw (4,1.5) -- (5,1.5);

    \draw (4,-1.5) -- (4,1.5);
    \draw (5,-1.5) -- (5,1.5);
    \node at (7,0) {$\to$};
    \draw (9,-2) -- (11,-2);
    \draw (9,-1) -- (11,-1);
    \draw (9,0) -- (11,0);
    \draw (9,1) -- (11,1);
    \draw (9,2) -- (11,2);

    \draw (9,-2) -- (9,2);
    \draw (10,-2) -- (10,2);
    \draw (11,-2) -- (11,2);
    \node at (13,0) {$\to$};
    \draw (15,-2) -- (17,-2);
    \draw (15,-1) -- (17,-1);
    \draw (15,0) -- (16,0);
    \draw (15,1) -- (16,1);
    \draw (15,2) -- (16,2);

    \draw (15,-2) -- (15,2);
    \draw (16,-2) -- (16,2);
    \draw (17,-2) -- (17,-1);
    \node at (19,0) {$\to$};
    \node at (21,0) {$0.$};
\end{tikzpicture}
\end{tabular}
\end{center}

Define $N = \Omega(V)$. Then $V = \Omega^{-1}(N)$. By Lemma~\ref{lem:syzygytensor}, we have that $V^{\otimes n} \cong (\Omega^{-1}(N))^{\otimes n}$ is isomorphic, up to projective summands, to $\Omega^{-n}(N^{\otimes n})$. By Lemma~\ref{lem:projReven}, all projective representations are even dimensional, so $V_n$, the odd-dimensional summand of $V^{\otimes n}$, is also the odd-dimensional summand of $\Omega^{-n}(N^{\otimes n})$. Thus, to find $V_n$, we must find the unique odd-dimensional indecomposable summand of $N^{\otimes n}$ and take $n$ cosyzygies.

We have that $N$ is the monomial representation given by the partition $(3)$. By Proposition~\ref{prop:symm180tenspow}, the odd-dimensional indecomposable summand of $N^{\otimes \text{odd}}$ is isomorphic to $N$, and the odd-dimensional indecomposable summand of $N^{\otimes \text{even}}$ is isomorphic to $\kk$. Thus computing $V_n$ reduces to computing the unique odd-dimensional summands of $\Omega^{-n}(N)$ for odd $n$ and $\Omega^{-n}(\kk)$ for even $n$, which is a much simpler task.

This technique relies on the fact that $\Omega(V)$ is a simpler representation than $V$. The same technique works for partitions of the form $(2^m, 1)$, where we get $P_V(n) = 2^m x + 1$. However, for the representations considered in Sections~\ref{subsec:graded311rep} and~\ref{subsec:graded42-1rep}, taking syzygies does not simplify the picture, and this technique cannot be applied.

\subsection{$(3,1,1)$ Monomial representation}\label{subsec:graded311rep}

Let $V$ be the monomial representation of $\mathbb{Z}/4\mathbb{Z} \times \mathbb{Z}/4\mathbb{Z}$ (or $\alpha(2,2)$) corresponding to the partition $(3,1,1)$.

The syzygy method in Section~\ref{subsec:gradedsyzygy} does not work for this monomial representation, since $\Omega(V)$ is not easier to analyze than $V$. However, we are still able to obtain an explicit decomposition, via a similar analysis as Section~\ref{subsec:graded41rep}, giving one of the first nontrivial examples where Benson's tensor powers conjecture may be verified.

\begin{proposition}
    We have the following decomposition into indecomposable summands:
    $$V_{2k-1}\otimes V = V_{2k} \oplus \underbrace{W_{12} \oplus \cdots \oplus W_{12}}_{k \text{\normalfont\ copies}} \oplus \underbrace{F \oplus \cdots \oplus F}_{3k-3 \text{\normalfont\ copies}} \oplus \underbrace{W_{28} \oplus \cdots \oplus W_{28}}_{k-1 \text{\normalfont\ copies}},$$
    $$V_{2k} \otimes V = V_{2k+1} \oplus \underbrace{W_{20} \oplus \cdots \oplus W_{20}}_{2k \text{\normalfont\ copies}},$$
    where $\dim W_{20} = 20, \dim W_{12} = 12, \dim W_{28}=28$, and $F$ is a free module of dimension $16$.
\end{proposition}

Again, we explicitly write the decomposition with a long but stronger lemma.

\begin{lemma}
We claim the following.
\begin{enumerate}[leftmargin=*]
    \item The representation $V_{2k+1}$ for $k\geq 1$ is given by the graded monomial diagram in Table~\ref{tab:311part-lem-1}.

    \item The representation $V_{2k}$ for $k \geq 1$ is given by the graded monomial diagram in Table~\ref{tab:311part-lem-2}.
    
    \item The representation $V_{2k-1} \otimes V$ decomposes into the following summands: $V_{2k}$ (given in Table~\ref{tab:311part-lem-3-V2k}); dimension $12$ summands (given in Table~\ref{tab:311part-lem-3-dim12}); dimension $16$ summands (given in Table~\ref{tab:311part-lem-3-dim16}); and dimension $28$ summands (given in Table~\ref{tab:311part-lem-3-dim28}).

    \item The representation $V_{2k}\otimes V$ decomposes into the following summands: $V_{2k+1}$ (given in Table~\ref{tab:311part-lem-4-V2k+1}) and dimension $20$ summands (given in Table~\ref{tab:311part-lem-4-dim20}).
\end{enumerate}
\end{lemma}

\begin{table}[H]
    \centering
    \caption{The pieces of $V_{2k+1}$. These pieces show the nonzero actions of $x$ and $y$, and they overlap each other. They are arranged as in Figure~\ref{fig:311part-lem-1} for $V_7$.}
    \label{tab:311part-lem-1}
    \scalebox{0.75}{
    \makebox[\textwidth][c]{
\begin{tabular}{|c|c|}
    \hline
    $(A,i)$ piece: & $(B,i)$ piece:
    \\
    \begin{tikzpicture}
        \draw (0,0) rectangle node{$v_{2k-1+2i,4k+1-2i}^{A,i}$} ++(3,0.95);
        \draw (0,0.95) rectangle node{$v_{2k-1+2i,4k+2-2i}^{A,i}$} ++(3,0.95);
        \draw (0,1.9) rectangle node{$v_{2k-1+2i,4k+3-2i}^{A,i}$} ++(3,0.95);
        \draw (0,2.85) rectangle node{$v_{2k-1+2i,4k+4-2i}^{A,i}$} ++(3,0.95);
    
        \draw (3,0) rectangle node{$v_{2k+2i,4k+1-2i}^{A,i}$} ++(3,0.95);
        \draw (3,0.95) rectangle node{$v_{2k+2i,4k+2-2i}^{A,i}$} ++(3,0.95);
        \draw (3,1.9) rectangle node{$v_{2k+2i,4k+3-2i}^{A,i}$} ++(3,0.95);
        \draw (3,2.85) rectangle node{$v_{2k+2i,4k+4-2i}^{A,i}$} ++(3,0.95);
    
        \draw (6,0) rectangle node{$v_{2k+1+2i,4k+1-2i}^{A,i}$} ++(3,0.95);
        \draw (6,0.95) rectangle node{$v_{2k+1+2i,4k+2-2i}^{A,i}$} ++(3,0.95);
        \draw (6,1.9) rectangle node{$v_{2k+1+2i,4k+3-2i}^{A,i}$} ++(3,0.95);
        \draw (6,2.85) rectangle node{$v_{2k+1+2i,4k+4-2i}^{A,i}$} ++(3,0.95);
    
        \draw (9,0) rectangle node{$v_{2k+2+2i,4k+1-2i}^{A,i}$} ++(3,0.95);
        \draw (9,0.95) rectangle node{$v_{2k+2+2i,4k+2-2i}^{A,i}$} ++(3,0.95);
        \draw (9,1.9) rectangle node{$v_{2k+2+2i,4k+3-2i}^{A,i}$} ++(3,0.95);
    \end{tikzpicture}
    &
    \begin{tikzpicture}
        \draw (0,0) rectangle node{$v_{2k-1+2i,4k+3-2i}^{B,i}$} ++(3,0.95);
        \draw (0,0.95) rectangle node{$v_{2k-1+2i,4k+4-2i}^{B,i}$} ++(3,0.95);
        \draw (0,1.9) rectangle node{$v_{2k-1+2i,4k+5-2i}^{B,i}$} ++(3,0.95);
        \draw (0,2.85) rectangle node{$v_{2k-1+2i,4k+6-2i}^{A,i-1}$} ++(3,0.95);
        \draw (3,0) rectangle node{$v_{2k+2i,4k+3-2i}^{B,i}$} ++(3,0.95);
        \draw (6,0) rectangle node{$v_{2k+1+2i,4k+3-2i}^{B,i}$} ++(3,0.95);
        \draw (9,0) rectangle node{$v_{2k+2+2i,4k+3-2i}^{A,i}$} ++(3,0.95);
    \end{tikzpicture}
    \\
    for $i=1,2,\dots,k$ & for $i=2,3,\dots,k$
    \\\hline
    $(B,1)$ piece: & $(B,k+1)$ piece:
    \\
\hline
    \end{tabular}
    }
    }
\end{table}

\begin{figure}[H]
    \centering
    \caption{The location of the pieces for $V_7$. The outlines show the $(A,i)$ pieces, and the shaded regions show the $(B,i)$ pieces.}
    \label{fig:311part-lem-1}
    %
\end{figure}

\begin{table}[H]
    \centering
    \caption{The pieces of $V_{2k}$. These pieces show the nonzero actions of $x$ and $y$, and they overlap each other. They are arranged as in Figure~\ref{fig:311part-lem-2} for $V_{6}$.}
    \label{tab:311part-lem-2}
    \scalebox{0.75}{
    \makebox[\textwidth][c]{
    \begin{tabular}{|c|c|}
    \hline
    $(A,i)$ piece: & $(B,i)$ piece:
    \\
    \begin{tikzpicture}
        \draw (0,0) rectangle node{$v_{2k-2+2i,4k-2i}^{i}$} ++(3,1);
        \draw (0,1) rectangle node{$v_{2k-2+2i,4k+1-2i}^{i}$} ++(3,1);
        \draw (0,2) rectangle node{$v_{2k-2+2i,4k+2-2i}^{i}$} ++(3,1);
        \draw (0,3) rectangle node{$v_{2k-2+2i,4k+3-2i}^{i}$} ++(3,1);
    
        \draw (3,0) rectangle node{$v_{2k-1+2i,4k-2i}^{i}$} ++(3,1);
        \draw (3,1) rectangle node{$v_{2k-1+2i,4k+1-2i}^{i}$} ++(3,1);
        \draw (3,2) rectangle node{$v_{2k-1+2i,4k+2-2i}^{i}$} ++(3,1);
    
        \draw (6,0) rectangle node{$v_{2k+2i,4k-2i}^{i}$} ++(3,1);
        \draw (6,1) rectangle node{$v_{2k+2i,4k+1-2i}^{i}$} ++(3,1);
        \draw (6,2) rectangle node{$v_{2k+2i,4k+2-2i}^{i}$} ++(3,1);
    
        \draw (9,0) rectangle node{$v_{2k+1+2i,4k-2i}^{i}$} ++(3,1);
    \end{tikzpicture}
    &
    \begin{tikzpicture}
        \draw (0,0) rectangle node{$v_{2k+2i,4k-2i}$} ++(3,1);
        \draw (0,1) rectangle node{$v_{2k+2i,4k+1-2i}^{i+1}$} ++(3,1);
        \draw (3,0) rectangle node{$v_{2k+1+2i,4k-2i}^{i}$} ++(3,1);
    \end{tikzpicture}
    \\
    for $i=1,2,\dots,k$ & for $i=1,2,\dots,k-1$
    \\\hline
    $(C,1)$ piece: & $(C,k)$ piece:
    \\
    \begin{tikzpicture}
        \draw (0,0) rectangle node{$v_{2k,4k}$} ++(2.5,1);
        \draw (0,1) rectangle node{$v_{2k,4k+1}^{1}$} ++(2.5,1);
    \end{tikzpicture}
    &
    \begin{tikzpicture}
        \draw (0,0) rectangle node{$v_{4k,2k}$} ++(2.5,1);
        \draw (2.5,0) rectangle node{$v_{4k+1,2k}^{k}$} ++(2.5,1);
    \end{tikzpicture}
    \\\hline
    \end{tabular}
    }
    }
\end{table}

\begin{figure}[H]
    \centering
    \caption{The location of the pieces for $V_6$. The outlines show the $(A,i)$ pieces, and the shaded regions show the $(B,i)$, $(C,1)$, and $(C,k)$ pieces.}
    \label{fig:311part-lem-2}
    %
\end{figure}

\begin{table}[H]
    \centering
    \caption{The pieces of $V_{2k}$.}
    \label{tab:311part-lem-3-V2k}
    \scalebox{0.75}{
    \makebox[\textwidth][c]{
    \begin{tabular}{|c|c|}
    \hline
    $(A,i)$ piece: & $(B,i)$ piece:
    \\
\hline
    \end{tabular}
    }
    }
\end{table}

\begin{table}[H]
    \centering
    \caption*{\textsc{Table~\ref{tab:311part-lem-3-V2k}} (continued).}
    \scalebox{0.75}{
    \makebox[\textwidth][c]{
    \begin{tabular}{|>{\centering\arraybackslash}p{17.2cm}|>{\centering\arraybackslash}p{8.8cm}|}
    \hline
    $(A,1)$ piece: & $(C,1)$ piece:
    \\
\hline
    \end{tabular}
    }
    }
\end{table}

\begin{table}[H]
    \centering
    \caption{The dimension $12$ summands. Note that when $i=1,k$, some of the vectors do not exist. This construction works with those nonexistent vectors removed.}
    \label{tab:311part-lem-3-dim12}
    \scalebox{0.75}{
    \makebox[\textwidth][c]{
    \begin{tabular}{|c|}
    \hline
    Dimension 12 summand:
    \\
\hline
    \end{tabular}
    }
    }
\end{table}

\begin{table}[H]
    \centering
    \caption{The dimension $16$ summands, which are all free modules. There are three families of them, which differ in only how they are expressed below.}
    \label{tab:311part-lem-3-dim16}
    \scalebox{0.75}{
    \makebox[\textwidth][c]{
    \begin{tabular}{|c|}
    \hline
    Dimension 16 Family 1 summand:
    \\
\hline
    \end{tabular}
    }
    }
\end{table}

\begin{table}[H]
    \centering
    \caption{The dimension $28$ summands, each of which is written in three pieces. These pieces show the nonzero actions of $x$ and $y$ and overlap each other. They are arranged as in Figure~\ref{fig:311part-lem-3-dim28}. Note that for $i=k-1$, some of the vectors do not exist. This construction works with those nonexistent vectors removed.}
    \label{tab:311part-lem-3-dim28}
    \scalebox{0.75}{
    \makebox[\textwidth][c]{
    \begin{tabular}{|c|}
    \hline
    Dimension 28 summand:
    \\
    \begin{tikzpicture}
        \draw (0,0) rectangle node{$v_{2k-3+2i,4k-2-2i}^{A,i}\otimes v_{1,2}$} ++(4.25,1);
        \draw (4.25,0) rectangle node{$v_{2k-2+2i,4k-2-2i}^{A,i}\otimes v_{1,2}$} ++(4.25,1);
        \draw (8.5,0) rectangle node{$v_{2k-1+2i,4k-2-2i}^{A,i}\otimes v_{1,2}$} ++(4.25,1);
        \draw (12.75,0) rectangle node{$v_{2k+2i,4k-2-2i}^{A,i}\otimes v_{1,2}$} ++(4.25,1);
        
        \draw (0,1) rectangle node{$\begin{aligned}&v_{2k-3+2i,4k-1-2i}^{A,i}\otimes v_{1,2}\\&+v_{2k-3+2i,4k-2-2i}^{A,i}\otimes v_{1,3}\end{aligned}$} ++(4.25,1.5);
        \draw (4.25,1) rectangle node{$\begin{aligned}&v_{2k-2+2i,4k-1-2i}^{A,i}\otimes v_{1,2}\\&+v_{2k-2+2i,4k-2-2i}^{A,i}\otimes v_{1,3}\end{aligned}$} ++(4.25,1.5);
        \draw (8.5,1) rectangle node{$\begin{aligned}&v_{2k-1+2i,4k-1-2i}^{A,i}\otimes v_{1,2}\\&+v_{2k-1+2i,4k-2-2i}^{A,i}\otimes v_{1,3}\end{aligned}$} ++(4.25,1.5);
        \draw (12.75,1) rectangle node{$\begin{aligned}&v_{2k+2i,4k-1-2i}^{A,i}\otimes v_{1,2}\\&+v_{2k+2i,4k-2-2i}^{A,i}\otimes v_{1,3}\end{aligned}$} ++(4.25,1.5);
    
        \draw (0,2.5) rectangle node{$v_{2k-3+2i,4k-2i}^{A,i}\otimes v_{1,2}$} ++(4.25,1);
        \draw (4.25,2.5) rectangle node{$v_{2k-2+2i,4k-2i}^{A,i}\otimes v_{1,2}$} ++(4.25,1);
        \draw (8.5,2.5) rectangle node{$v_{2k-1+2i,4k-2i}^{A,i}\otimes v_{1,2}$} ++(4.25,1);
    
        \draw (0,3.5) rectangle node{$v_{2k-3+2i,4k-2i}^{A,i}\otimes v_{1,3}$} ++(4.25,1);
        \draw (4.25,3.5) rectangle node{$v_{2k-2+2i,4k-2i}^{A,i}\otimes v_{1,3}$} ++(4.25,1);
        \draw (8.5,3.5) rectangle node{$v_{2k-1+2i,4k-2i}^{A,i}\otimes v_{1,3}$} ++(4.25,1);
    \end{tikzpicture}
    \\
    \\
\hline
    \end{tabular}
    }
    }
\end{table}

\begin{figure}[H]
    \centering
    \caption{The location of the pieces for the dimension $28$ summand. The outlines show the positions of the two pieces with $14$ cells, and the shaded region shows the position of the piece with $4$ cells.}
    \label{fig:311part-lem-3-dim28}
    %
\end{figure}

\begin{table}[H]
    \centering
    \caption{The pieces of $V_{2k+1}$.}
    \label{tab:311part-lem-4-V2k+1}
    \scalebox{0.75}{
    \makebox[\textwidth][c]{
    \begin{tabular}{|c|c|}
    \hline
    \multicolumn{2}{|c|}{$(A,i)$ piece:}
    \\
\hline
    \end{tabular}
    }
    }
\end{table}

\begin{table}[H]
    \centering
    \caption*{\textsc{Table~\ref{tab:311part-lem-4-V2k+1}} (continued).}
    \scalebox{0.75}{
    \makebox[\textwidth][c]{
    \begin{tabular}{|>{\centering\arraybackslash}p{24.7cm}|}
    \hline
    $(B,i)$ piece:
    \\
\hline
    \end{tabular}
    }
    }
\end{table}
\vspace{-2mm}
\begin{table}[H]
    \centering
    \caption{The dimension $20$ summands, which can be split into two orientations, each orientation being characterized in two pieces which overlap at the indicated equality. The pieces are arranged as in Figure~\ref{fig:311part-lem-4-dim20}.}
    \label{tab:311part-lem-4-dim20}
    \scalebox{0.75}{
    \makebox[\textwidth][c]{
    \begin{tabular}{|c|}
    \hline
    Vertical orientation:
    \\
\hline
    \end{tabular}
    }
    }
    
\end{table}

\begin{table}[H]
    \centering
    \caption*{\textsc{Table~\ref{tab:311part-lem-4-dim20}} (continued).}
    \scalebox{0.75}{
    \makebox[\textwidth][c]{
    \begin{tabular}{|>{\centering\arraybackslash}p{25.4cm}|}
    \hline
    Horizontal orientation:
    \\
\hline
    \end{tabular}
    }
    }
\end{table}

\begin{figure}[H]
    \centering
    \caption{The location of the pieces for the dimension $20$ summands. The outlines show the positions of the pieces with $17$ cells, and the shaded region shows the position of the piece with $4$ cells.}
    \label{fig:311part-lem-4-dim20}
    \begin{tabular}{cc}
    \begin{tikzpicture}[scale=0.4]
        \fill[lightgray] (2-0.5,0-0.5) rectangle ++(2,3);
        \fill[lightgray] (2-0.5,0-0.5) rectangle ++(4,1);
    
        \node at (2,0) {$1$};
        \node at (3,0) {$1$};
        \node at (4,0) {$1$};
        \node at (5,0) {$1$};
    
        \node at (0,1) {$1$};
        \node at (1,1) {$1$};
        \node at (2,1) {$2$};
        \node at (3,1) {$1$};
    
        \node at (0,2) {$1$};
        \node at (1,2) {$1$};
        \node at (2,2) {$2$};
        \node at (3,2) {$1$};
    
        \node at (0,3) {$1$};
        \node at (1,3) {$1$};
        \node at (2,3) {$1$};
    
        \node at (0,4) {$1$};
        \node at (1,4) {$1$};
        \node at (2,4) {$1$};
    
        \draw[thick] (2.5,-0.5) -- (5.5,-0.5) -- (5.5,0.5) -- (3.5,0.5) -- (3.5,2.5) -- (2.5,2.5) -- (2.5,4.5) -- (-0.5,4.5) -- (-0.5,0.5) -- (2.5,0.5) -- (2.5,-0.5);
    \end{tikzpicture}
    &
    \begin{tikzpicture}[scale=0.4]
        \fill[lightgray] (0-0.5,2-0.5) rectangle ++(3,2);
        \fill[lightgray] (0-0.5,2-0.5) rectangle ++(1,4);
    
        \node at (1,0) {$1$};
        \node at (2,0) {$1$};
        \node at (3,0) {$1$};
        \node at (4,0) {$1$};
    
        \node at (1,1) {$1$};
        \node at (2,1) {$1$};
        \node at (3,1) {$1$};
        \node at (4,1) {$1$};
    
        \node at (0,2) {$1$};
        \node at (1,2) {$2$};
        \node at (2,2) {$2$};
        \node at (3,2) {$1$};
        \node at (4,2) {$1$};
    
        \node at (0,3) {$1$};
        \node at (1,3) {$1$};
        \node at (2,3) {$1$};
    
        \node at (0,4) {$1$};
        \node at (0,5) {$1$};
    
        \draw[thick] (-0.5,2.5) -- (-0.5,5.5) -- (0.5,5.5) -- (0.5,3.5) -- (2.5,3.5) -- (2.5,2.5) -- (4.5,2.5) -- (4.5,-0.5) -- (0.5,-0.5) -- (0.5,2.5) -- (-0.5,2.5);
    \end{tikzpicture}
    \end{tabular}
\end{figure}

Now, the proof is very similar to the $(4,1)$ partition case, which was Lemma~\ref{lem:41partition}.

\begin{proof}
Similar to the proof of Lemma~\ref{lem:41partition}, the explicit constructions show that the summands are both subrepresentations and indecomposable. The positions of the summands of $V_{2k-1} \otimes V$ are summarized in Figure~\ref{fig:311par-even-location}, and the positions of the summands of $V_{2k}\otimes V$ are summarized in Figure~\ref{fig:311par-odd-location} below. The black outline is the location of $V_i$ within the tensor product. 

It can be checked that these summands span the tensor product. This is very similar to Lemma~\ref{lem:41partition}, and we have omitted the proof.
\end{proof}

 As a corollary, we observe that Conjecture~\ref{conj:tensorpowers} again holds in this case; this time, the function $P_V(n)$ is not polynomial, but is still quasi-polynomial.

\begin{corollary}\label{3-1-1-poly}
The function $P_V(n)$ is a quasi-polynomial with period $2$, given by $10n-5$ for odd $n$ and $6n+1$ for even $n$.
\end{corollary}

\begin{figure}[H]
    \centering
    \caption{Positions of the indecomposable summands in $V_7 \otimes V$.}
    \label{fig:311par-even-location}
    \begin{subfigure}[H]{0.25\textwidth}
        \centering
        \scalebox{0.9}{
        \begin{tikzpicture}[scale=0.5]
            \foreach \i in {0,2,4} {
                \draw[semithick] (-\i+4-0.5,\i-0.5) rectangle ++(1,1);
                \node[draw, shape=diamond, semithick] at (-\i+5,\i) {};
                \node[draw, fill=lightgray, shape=diamond, semithick] at (-\i+4,1+\i) {};
            }
            
            \node at (4,0) {$1$};
            \node at (5,0) {$2$};
            \node at (6,0) {$4$};
            \node at (7,0) {$5$};
            \node at (8,0) {$5$};
            \node at (9,0) {$3$};
            \node at (10,0) {$1$};
        
            \node at (4,1) {$2$};
            \node at (5,1) {$3$};
            \node at (6,1) {$6$};
            \node at (7,1) {$6$};
            \node at (8,1) {$4$};
            \node at (9,1) {$1$};
           
            \node at (2,2) {$1$};
            \node at (3,2) {$2$};
            \node at (4,2) {$7$};
            \node at (5,2) {$9$};
            \node at (6,2) {$13$};
            \node at (7,2) {$10$};
            \node at (8,2) {$5$};
            \node at (9,2) {$1$};
        
            \node at (2,3) {$2$};
            \node at (3,3) {$3$};
            \node at (4,3) {$9$};
            \node at (5,3) {$10$};
            \node at (6,3) {$11$};
            \node at (7,3) {$5$};
            \node at (8,3) {$1$};
         
            \node at (0,4) {$1$};
            \node at (1,4) {$2$};
            \node at (2,4) {$7$};
            \node at (3,4) {$9$};
            \node at (4,4) {$15$};
            \node at (5,4) {$12$};
            \node at (6,4) {$8$};
            \node at (7,4) {$2$};
        
            \node at (0,5) {$2$};
            \node at (1,5) {$3$};
            \node at (2,5) {$9$};
            \node at (3,5) {$10$};
            \node at (4,5) {$12$};
            \node at (5,5) {$6$};
            \node at (6,5) {$2$};
        
            \node at (0,6) {$4$};
            \node at (1,6) {$6$};
            \node at (2,6) {$13$};
            \node at (3,6) {$11$};
            \node at (4,6) {$8$};
            \node at (5,6) {$2$};
        
            \node at (0,7) {$5$};
            \node at (1,7) {$6$};
            \node at (2,7) {$10$};
            \node at (3,7) {$5$};
            \node at (4,7) {$2$};
        
            \node at (0,8) {$5$};
            \node at (1,8) {$4$};
            \node at (2,8) {$5$};
            \node at (3,8) {$1$};
        
            \node at (0,9) {$3$};
            \node at (1,9) {$1$};
            \node at (2,9) {$1$};
            
            \node at (0,10) {$1$};
        
            \draw[thick] (5.5,-0.5) -- (9.5,-0.5) -- (9.5,0.5) -- (8.5,0.5) -- (8.5,2.5) -- (6.5,2.5) -- (6.5,4.5) -- (4.5,4.5) -- (4.5,6.5) -- (2.5,6.5) -- (2.5,8.5) -- (0.5,8.5) -- (0.5,9.5) -- (-0.5,9.5) -- (-0.5,5.5) -- (1.5,5.5) -- (1.5,3.5) -- (3.5,3.5) -- (3.5,1.5) -- (5.5,1.5) -- (5.5,-0.5);
        
            \foreach \i in {0,2,4,6}
                \draw (6-\i,\i) circle (0.5);
            \foreach \i in {0,2,4}
                \draw[very thick] (6-0.35-\i,0.35+\i) -- (4+0.35-\i,2-0.35+\i);
        
            \foreach \i in {0,...,7}
                \draw (7-\i,\i) circle (0.5);
            \foreach \i in {0,2,4,6} {
                \draw[very thick] (7-\i,\i) -- (6-\i,1+\i);
            }
        \end{tikzpicture}
        }
    \end{subfigure}
    \hfill
    \begin{subtable}[H]{0.65\textwidth}
        \centering
        \caption*{Key}
        \scalebox{0.9}{
        \begin{tabular}{|c|c|c|c|c|}
            \hline
            Dim. $12$ & Family 1 & Family 2 & Family 3 & Dim. $28$
            \\\hline
            \begin{tikzpicture}[scale=0.4]
                \node at (1,0) {$1$};
                \node at (2,0) {$1$};
                \node at (3,0) {$1$};
                \node at (4,0) {$1$};
            
                \node at (0,1) {$1$};
                \node at (1,1) {$1$};
                \node at (2,1) {$1$};
            
                \node at (0,2) {$1$};
                \node at (1,2) {$1$};
                \node at (2,2) {$1$};
            
                \node at (0,3) {$1$};
                \node at (0,4) {$1$};
            
                \draw (1,0) circle (0.5);
                \draw (0,1) circle (0.5);
                \draw[very thick] (1,0) -- (0,1);
            \end{tikzpicture}
            &
            \begin{tikzpicture}[scale=0.4]
                \node at (0,0) {$1$};
                \node at (1,0) {$1$};
                \node at (2,0) {$1$};
                \node at (3,0) {$1$};
                \node at (0,1) {$1$};
                \node at (1,1) {$1$};
                \node at (2,1) {$1$};
                \node at (3,1) {$1$};
                \node at (0,2) {$1$};
                \node at (1,2) {$1$};
                \node at (2,2) {$1$};
                \node at (3,2) {$1$};
                \node at (0,3) {$1$};
                \node at (1,3) {$1$};
                \node at (2,3) {$1$};
                \node at (3,3) {$1$};
                
                \draw[semithick] (0-0.5,0-0.5) rectangle ++(1,1);
            \end{tikzpicture}
            &
            \begin{tikzpicture}[scale=0.4]
                \node at (0,0) {$1$};
                \node at (1,0) {$1$};
                \node at (2,0) {$1$};
                \node at (3,0) {$1$};
                \node at (0,1) {$1$};
                \node at (1,1) {$1$};
                \node at (2,1) {$1$};
                \node at (3,1) {$1$};
                \node at (0,2) {$1$};
                \node at (1,2) {$1$};
                \node at (2,2) {$1$};
                \node at (3,2) {$1$};
                \node at (0,3) {$1$};
                \node at (1,3) {$1$};
                \node at (2,3) {$1$};
                \node at (3,3) {$1$};
            
                \node[draw, shape=diamond, semithick] (0,0) {};
            \end{tikzpicture}
            &
            \begin{tikzpicture}[scale=0.4]
                \node[draw, fill=lightgray, shape=diamond, semithick] (0,0) {};
            
                \node at (0,0) {$1$};
                \node at (1,0) {$1$};
                \node at (2,0) {$1$};
                \node at (3,0) {$1$};
                \node at (0,1) {$1$};
                \node at (1,1) {$1$};
                \node at (2,1) {$1$};
                \node at (3,1) {$1$};
                \node at (0,2) {$1$};
                \node at (1,2) {$1$};
                \node at (2,2) {$1$};
                \node at (3,2) {$1$};
                \node at (0,3) {$1$};
                \node at (1,3) {$1$};
                \node at (2,3) {$1$};
                \node at (3,3) {$1$};
            \end{tikzpicture}
            &
            \begin{tikzpicture}[scale=0.4]
                \node at (2,0) {$1$};
                \node at (3,0) {$1$};
                \node at (4,0) {$1$};
                \node at (5,0) {$1$};
            
                \node at (2,1) {$1$};
                \node at (3,1) {$1$};
                \node at (4,1) {$1$};
                \node at (5,1) {$1$};
            
                \node at (0,2) {$1$};
                \node at (1,2) {$1$};
                \node at (2,2) {$3$};
                \node at (3,2) {$2$};
                \node at (4,2) {$1$};
                \node at (5,2) {$1$};
            
                \node at (0,3) {$1$};
                \node at (1,3) {$1$};
                \node at (2,3) {$2$};
                \node at (3,3) {$1$};
            
                \node at (0,4) {$1$};
                \node at (1,4) {$1$};
                \node at (2,4) {$1$};
            
                \node at (0,5) {$1$};
                \node at (1,5) {$1$};
                \node at (2,5) {$1$};
            
                \draw (2,0) circle (0.5);
                \draw (0,2) circle (0.5);
                \draw[very thick] (2-0.35,0.35) -- (0.35,2-0.35);
            \end{tikzpicture}
            \\\hline
            \end{tabular}
            }
    \end{subtable}
\end{figure}

\begin{figure}[H]
    \centering
    \caption{Positions of the indecomposable summands in $V_8 \otimes V$.}
    \label{fig:311par-odd-location}
    \begin{subfigure}[H]{0.45\textwidth}
        \centering
        \scalebox{0.9}{
        \begin{tikzpicture}[scale=0.5]
            \node at (6,0) {$1$};
            \node at (7,0) {$2$};
            \node at (8,0) {$4$};
            \node at (9,0) {$4$};
            \node at (10,0) {$3$};
            \node at (11,0) {$1$};
            
            \node at (6,1) {$2$};
            \node at (7,1) {$3$};
            \node at (8,1) {$5$};
            \node at (9,1) {$3$};
            \node at (10,1) {$1$};
            
            \node at (4,2) {$1$};
            \node at (5,2) {$2$};
            \node at (6,2) {$7$};
            \node at (7,2) {$8$};
            \node at (8,2) {$9$};    
            \node at (9,2) {$4$};
            \node at (10,2) {$1$};
        
            \node at (4,3) {$2$};
            \node at (5,3) {$3$};
            \node at (6,3) {$8$};
            \node at (7,3) {$6$};
            \node at (8,3) {$4$};
            
            \node at (2,4) {$1$};
            \node at (3,4) {$2$};
            \node at (4,4) {$7$};
            \node at (5,4) {$8$};
            \node at (6,4) {$11$};
            \node at (7,4) {$5$};
            \node at (8,4) {$2$};
            
            \node at (2,5) {$2$};
            \node at (3,5) {$3$};
            \node at (4,5) {$8$};
            \node at (5,5) {$6$};
            \node at (6,5) {$5$};
            
            \node at (0,6) {$1$};
            \node at (1,6) {$2$};
            \node at (2,6) {$7$};
            \node at (3,6) {$8$};
            \node at (4,6) {$11$};
            \node at (5,6) {$5$};
            \node at (6,6) {$2$};
            
            \node at (0,7) {$2$};
            \node at (1,7) {$3$};
            \node at (2,7) {$8$};
            \node at (3,7) {$6$};
            \node at (4,7) {$5$};
            
            \node at (0,8) {$4$};
            \node at (1,8) {$5$};
            \node at (2,8) {$9$};
            \node at (3,8) {$4$};
            \node at (4,8) {$2$};
              
            \node at (0,9) {$4$};
            \node at (1,9) {$3$};
            \node at (2,9) {$4$};
         
            \node at (0,10) {$3$};
            \node at (1,10) {$1$};
            \node at (2,10) {$1$};
            
            \node at (0,11) {$1$};
        
            \draw[thick] (5.5,-0.5) -- (10.5,-0.5) -- (10.5,0.5) -- (9.5,0.5) -- (9.5,2.5) -- (8.5,2.5) -- (8.5,3.5) -- (7.5,3.5) -- (7.5,4.5) -- (6.5,4.5) -- (6.5,5.5) -- (5.5,5.5) -- (5.5,6.5) -- (4.5,6.5) -- (4.5,7.5) -- (3.5,7.5) -- (3.5,8.5) -- (2.5,8.5) -- (2.5,9.5) -- (0.5,9.5) -- (0.5,10.5) -- (-0.5,10.5) -- (-0.5,5.5) -- (1.5,5.5) -- (1.5,3.5) -- (3.5,3.5) -- (3.5,1.5) -- (5.5,1.5) -- (5.5,-0.5);
            
            \foreach \i in {0,...,7}
                \draw (7-\i,\i) circle (0.5);
            \foreach \i in {0,...,4}
                \draw (8-2*\i,2*\i) circle (0.5);
        
            \foreach \i in {0,1,2,3} {
                \draw[very thick] (8-2*\i,2*\i) -- (6-2*\i,1+2*\i);
                \draw[very thick] (7-2*\i,2*\i) -- (6-2*\i,2+2*\i);
            }
        \end{tikzpicture}
        }
    \end{subfigure}
    \hfill
    \begin{subtable}[H]{0.45\textwidth}
        \centering
        \caption*{Key}
        \scalebox{0.9}{
        \begin{tabular}{|c|c|}
            \hline
            Vertical & Horizontal
            \\\hline
            \begin{tikzpicture}[scale=0.4]
                \node at (2,0) {$1$};
                \node at (3,0) {$1$};
                \node at (4,0) {$1$};
                \node at (5,0) {$1$};
            
                \node at (0,1) {$1$};
                \node at (1,1) {$1$};
                \node at (2,1) {$2$};
                \node at (3,1) {$1$};
            
                \node at (0,2) {$1$};
                \node at (1,2) {$1$};
                \node at (2,2) {$2$};
                \node at (3,2) {$1$};
            
                \node at (0,3) {$1$};
                \node at (1,3) {$1$};
                \node at (2,3) {$1$};
            
                \node at (0,4) {$1$};
                \node at (1,4) {$1$};
                \node at (2,4) {$1$};
            
                \draw (2,0) circle (0.5);
                \draw (0,1) circle (0.5);
            
                \draw[very thick] (2,0) -- (0,1);
            \end{tikzpicture}
            &
            \begin{tikzpicture}[scale=0.4]
                \node at (1,0) {$1$};
                \node at (2,0) {$1$};
                \node at (3,0) {$1$};
                \node at (4,0) {$1$};
            
                \node at (1,1) {$1$};
                \node at (2,1) {$1$};
                \node at (3,1) {$1$};
                \node at (4,1) {$1$};
            
                \node at (0,2) {$1$};
                \node at (1,2) {$2$};
                \node at (2,2) {$2$};
                \node at (3,2) {$1$};
                \node at (4,2) {$1$};
            
                \node at (0,3) {$1$};
                \node at (1,3) {$1$};
                \node at (2,3) {$1$};
            
                \node at (0,4) {$1$};
                \node at (0,5) {$1$};
            
                \draw (1,0) circle (0.5);
                \draw (0,2) circle (0.5);
            
                \draw[very thick] (1,0) -- (0,2);
            \end{tikzpicture}
            \\\hline
        \end{tabular}
        }
    \end{subtable}
\end{figure}

\subsection{$(4,2)/(1)$ Monomial representation}\label{subsec:graded42-1rep}

Let $V$ be the monomial representation of $\mathbb{Z}/2\mathbb{Z} \times \mathbb{Z}/4\mathbb{Z}$ (or $\alpha(1,2)$) corresponding to the skew partition $(4,2)/(1)$. This is another example where the syzygy technique outlined in Section \ref{subsec:gradedsyzygy} is not applicable. 

\begin{proposition}
    We have the following decomposition into indecomposable summands:
    $$V_{2k-1}\otimes V = V_{2k} \oplus \underbrace{F\oplus \cdots \oplus F}_{3k-3 \text{\normalfont\ copies}} \oplus \underbrace{W_{12} \oplus \cdots \oplus W_{12}}_{2k-1 \text{\normalfont\ copies}},$$
    $$V_{2k} \otimes V = V_{2k+1} \oplus \underbrace{F\oplus \cdots \oplus F}_{3k \text{\normalfont\ copies}}\oplus \underbrace{W_{12} \oplus \cdots \oplus W_{12}}_{2k \text{\normalfont\ copies}},$$
    where $\dim W_{12} = 12$ and $F$ is a free module of dimension $8$.
\end{proposition}

As before, we write the decomposition with a stronger lemma.

\begin{lemma}\label{lem:421part}
We claim the following.
\begin{enumerate}[leftmargin=*]
    \item The representation $V_{2k+1}$ for $k \geq 1$ is given by the graded monomial diagram in Table~\ref{tab:421part-lem-1}.

    \item The representation $V_{2k}$ for $k\geq 1$ is given by the graded monomial diagram in Table~\ref{tab:421part-lem-2}.

    \item The representation $V_{2k-1}\otimes V$ decomposes into the following summands: $V_{2k}$ (given in Table~\ref{tab:421part-lem-3-V2k}); dimension $8$ summands (given in Table~\ref{tab:421part-lem-3-dim8}); and dimension $12$ summands (given in Table~\ref{tab:421part-lem-3-dim12}).

    \item The representation $V_{2k}\otimes V$ decomposes into the following summands: $V_{2k+1}$ (given in Table~\ref{tab:421part-lem-4-V2k+1}); dimension $8$ summands (given in Table~\ref{tab:421part-lem-4-dim8}); and dimension $12$ summands (given in Table~\ref{tab:421part-lem-4-dim12}). 
\end{enumerate}
\end{lemma}

\begin{table}[H]
    \centering
    \caption{The pieces of $V_{2k+1}$. These pieces show the nonzero actions of $x$ and $y$, and they overlap each other. They are arranged as in Figure~\ref{fig:421part-lem-1} for $V_5$.}
    \label{tab:421part-lem-1}
    \vspace{-2mm}
    \scalebox{0.75}{
    \makebox[\textwidth][c]{
\begin{tabular}{|c|c|}
    \hline
    $(A,1)$ piece: & $(A,i)$ piece:
    \\
    \begin{tikzpicture}
        \draw (3,0) rectangle node{$v_{4k+2,2k+1}^{A,1}$} ++(3,0.8);
    
        \draw (0,0.8) rectangle node{$v_{4k+1,2k+2}^{A,1}$} ++(3,0.8);
        \draw (3,0.8) rectangle node{$v_{4k+2,2k+2}^{A,1}$} ++(3,0.8);
    
        \draw (0,1.6) rectangle node{$v_{4k+1,2k+3}^{A,1}$} ++(3,0.8);
        \draw (0,2.4) rectangle node{$v_{4k+1,2k+4}^{A,1}$} ++(3,0.8);
        \draw (0,3.2) rectangle node{$v_{4k+1,2k+5}^{A,1}$} ++(3,0.8);
    \end{tikzpicture}
    &
    \begin{tikzpicture}
        \draw (3,0) rectangle node{$v_{4k+4-2i,2k-3+4i}^{A,i}$} ++(3,0.8);
        \draw (6,0) rectangle node{$v_{4k+5-2i,2k-3+4i}^{A,i-1}$} ++(3,0.8);
    
        \draw (0,0.8) rectangle node{$v_{4k+3-2i,2k-2+4i}^{A,i}$} ++(3,0.8);
        \draw (3,0.8) rectangle node{$v_{4k+4-2i,2k-2+4i}^{A,i}$} ++(3,0.8);
    
        \draw (0,1.6) rectangle node{$v_{4k+3-2i,2k-1+4i}^{A,i}$} ++(3,0.8);
        \draw (0,2.4) rectangle node{$v_{4k+3-2i,2k+4i}^{A,i}$} ++(3,0.8);
        \draw (0,3.2) rectangle node{$v_{4k+3-2i,2k+1+4i}^{A,i}$} ++(3,0.8);
    \end{tikzpicture}
    \\
    & for $i=2,3,\dots,k$ 
    \\\hline
    $(B,i)$ piece: & $(A,k+1)$ piece:
    \\
\hline
    \end{tabular}
    }
    }
\end{table}
\vspace{-2mm}
\begin{figure}[H]
    \centering
    \caption{The location of the pieces for $V_5$. The outlines show the $(B,i)$ pieces, and the shaded regions show the $(A,i)$ pieces.}
    \label{fig:421part-lem-1}
%
\end{figure}
\vspace{-2mm}
\begin{table}[H]
    \centering
    \caption{The pieces of $V_{2k}$. These pieces show the nonzero actions of $x$ and $y$, and they overlap each other. They are arranged as in Figure~\ref{fig:421part-lem-2} for $V_6$.}
    \label{tab:421part-lem-2}
    \vspace{-2mm}
    \scalebox{0.75}{
    \makebox[\textwidth][c]{
    \begin{tabular}{|c|c|}
    \hline
    $(A,1)$ piece: & $(A,i)$ piece:
    \\
    \begin{tikzpicture}
        \draw (3,0) rectangle node{$v_{4k,2k}^{A,1}$} ++(3,0.8);
    
        \draw (0,0.8) rectangle node{$v_{4k-1,2k+1}^{A,1}$} ++(3,0.8);
        \draw (3,0.8) rectangle node{$v_{4k,2k+1}^{A,1}$} ++(3,0.8);
    
        \draw (0,1.6) rectangle node{$v_{4k-1,2k+2}^{A,1}$} ++(3,0.8);
        \draw (0,2.4) rectangle node{$v_{4k-1,2k+3}^{A,1}$} ++(3,0.8);
        \draw (0,3.2) rectangle node{$v_{4k-1,2k+4}^{A,1}$} ++(3,0.8);
    \end{tikzpicture}
    &
    \begin{tikzpicture}
        \draw (3,0) rectangle node{$v_{4k+2-2i,2k-4+4i}^{A,i}$} ++(3,0.8);
        \draw (6,0) rectangle node{$v_{4k+3-2i,2k-4+4i}^{A,i-1}$} ++(3,0.8);
    
        \draw (0,0.8) rectangle node{$v_{4k+1-2i,2k-3+4i}^{A,i}$} ++(3,0.8);
        \draw (3,0.8) rectangle node{$v_{4k+2-2i,2k-3+4i}^{A,i}$} ++(3,0.8);
    
        \draw (0,1.6) rectangle node{$v_{4k+1-2i,2k-2+4i}^{A,i}$} ++(3,0.8);
        \draw (0,2.4) rectangle node{$v_{4k+1-2i,2k-1+4i}^{A,i}$} ++(3,0.8);
        \draw (0,3.2) rectangle node{$v_{4k+1-2i,2k+4i}^{A,i}$} ++(3,0.8);
    \end{tikzpicture}
    \\
    & for $i=2,3,\dots,k$
    \\\hline
    $C$ piece: & $(B,i)$ piece:
    \\
    \begin{tikzpicture}
        \draw (0,0) rectangle node{$v_{2k,6k}$} ++(3,1);
        \draw (0,1) rectangle node{$v_{2k,6k+1}^{B,k}$} ++(3,1);
    \end{tikzpicture}
    &
    \begin{tikzpicture}
        \draw (0,0) rectangle node{$v_{4k-2i,2k-2+4i}^{B,i}$} ++(3,0.8);
        \draw (3,0) rectangle node{$v_{4k+1-2i,2k-2+4i}^{B,i}$} ++(3,0.8);
    
        \draw (0,0.8) rectangle node{$v_{4k-2i,2k-1+4i}^{B,i}$} ++(3,0.8);
        \draw (3,0.8) rectangle node{$v_{4k+1-2i,2k-1+4i}^{B,i}$} ++(3,0.8);
    
        \draw (0,1.6) rectangle node{$v_{4k-2i,2k+4i}^{B,i}$} ++(3,0.8);
        \draw (3,1.6) rectangle node{$v_{4k+1-2i,2k+4i}^{A,i}$} ++(3,0.8);
    
        \draw (0,2.4) rectangle node{$v_{4k-2i,2k+1+4i}^{B,i}$} ++(3,0.8);
    \end{tikzpicture}
    \\
\hline
    \end{tabular}
    }
    }
\end{table}

\begin{figure}[H]
    \centering
    \caption{The location of the pieces for $V_6$. The outlines show the $(B,i)$ pieces and the shaded regions show the $(A,i)$ and $C$ pieces.}
    \label{fig:421part-lem-2}
%
\end{figure}

\begin{table}[H]
    \centering
    \caption{The pieces of $V_{2k}$. Note that in pieces $(A,k)$ and $(B,k)$, some of the vectors do not exist. This construction works with those nonexistent vectors removed.}
    \label{tab:421part-lem-3-V2k}
    \scalebox{0.75}{
    \makebox[\textwidth][c]{
    \begin{tabular}{|c|c|}
    \hline
    $(A,i)$ piece: & $(B,i)$ piece:
    \\
\hline
    \end{tabular}
    }
    }    
\end{table}

\begin{table}[H]
    \centering
    \caption*{\textsc{Table~\ref{tab:421part-lem-3-V2k}} (continued).}
    \scalebox{0.75}{
    \makebox[\textwidth][c]{
    \begin{tabular}{|>{\centering\arraybackslash}p{13cm}|>{\centering\arraybackslash}p{8.8cm}|}
    \hline
    $(A,1)$ piece: & $C$ piece:
    \\
    \begin{tikzpicture}
        \draw (3.5,0) rectangle node{$v_{4k-2,2k-1}^{A,1}\otimes v_{2,1}$} ++(3.5,0.8);
        
        \draw (0,0.8) rectangle node{$\begin{aligned}&v_{4k-3,2k}^{A,1}\otimes v_{2,1}\\&+v_{4k-2,2k-1}^{A,1}\otimes v_{1,2}\end{aligned}$} ++(3.5,1.4);
        \draw (3.5,0.8) rectangle node{$\begin{aligned}&v_{4k-2,2k}^{A,1}\otimes v_{2,1}\\&+v_{4k-2,2k-1}^{A,1}\otimes v_{2,2}\end{aligned}$} ++(3.5,1.4);
    
        \draw (0,2.2) rectangle node{$\begin{aligned}&v_{4k-3,2k+1}^{A,1}\otimes v_{2,1}\\&+v_{4k-3,2k}^{A,1}\otimes v_{2,2}\\&+v_{4k-2,2k}^{A,1}\otimes v_{1,2}\\&+v_{4k-2,2k-1}^{A,1}\otimes v_{1,3}\end{aligned}$} ++(3.5,2.7);
        \draw (0,4.9) rectangle node{$\begin{aligned}&v_{4k-3,2k+2}^{A,1}\otimes v_{2,1}\\&+v_{4k-2,2k-1}^{A,1}\otimes v_{1,4}\end{aligned}$} ++(3.5,1.4);
        \draw (0,6.3) rectangle node{$\begin{aligned}&v_{4k-3,2k+3}^{A,1}\otimes v_{2,1}\\&+v_{4k-3,2k+2}^{A,1}\otimes v_{2,2}\\&+v_{4k-2,2k}^{A,1}\otimes v_{1,4}\end{aligned}$} ++(3.5,2.05);
    \end{tikzpicture}
    &
    \begin{tikzpicture}
        \draw (0,0) rectangle node{$v_{2k-1,6k-3}^{A,k}\otimes v_{1,3}$} ++(3.5,1);
        \draw (0,1) rectangle node{$\begin{aligned}&v_{2k-1,6k-2}^{A,k}\otimes v_{1,3}\\&+v_{2k-1,6k-3}^{A,k}\otimes v_{1,4}\end{aligned}$} ++(3.5,1.5);
    \end{tikzpicture}
    \\\hline
    \end{tabular}
    }
    }
\end{table}

\begin{table}[H]
    \centering
    \caption{The dimension $8$ summands, which are all free modules. There are three families of them, which differ in only how they are expressed below.}
    \label{tab:421part-lem-3-dim8}
    \scalebox{0.75}{
    \makebox[\textwidth][c]{
    \begin{tabular}{|c|c|c|}
    \hline
    Dimension $8$ Family $1$ summand: & Dimension $8$ Family $2$ summand: & Dimension $8$ Family $3$ summand:
    \\
\hline
    \end{tabular}
    }
    }
\end{table}

\begin{figure}[H]
    \centering
    \caption{The location of the pieces for the dimension $12$ summands. The outlines show the positions of the piece with $9$ cells, and the shaded region shows the position of the overlapping piece with $8$ cells.}
    \label{fig:421part-lem-3-dim12}
%
\end{figure}

\begin{table}[H]
    \centering
    \caption{The dimension $12$ summands, which can be split into two families (which differ in only how they are expressed). Each family is characterized in two pieces which overlap at the indicated equality. The pieces are arranged as in Figure~\ref{fig:421part-lem-3-dim12}. In the dimension $12$ Family $1$ summand for $i=k$, some of the vectors do not exist. This construction works with those nonexistent vectors removed.}
    \label{tab:421part-lem-3-dim12}
    \scalebox{0.75}{
    \makebox[\textwidth][c]{
    \begin{tabular}{|c||c|}
    \hline
    Dimension $12$ Family $1$ summand: & Dimension $12$ Family $2$ summand:
    \\
\hline
    \end{tabular}
    }
    }
\end{table}

\begin{table}[H]
    \centering
    \caption{The pieces of $V_{2k+1}$.}
    \label{tab:421part-lem-4-V2k+1}
    \scalebox{0.75}{
    \makebox[\textwidth][c]{
    \begin{tabular}{|c|c|}
    \hline
    $(A,i)$ piece: & $(B,i)$ piece:
    \\
\hline
    \end{tabular}
    }
    }
\end{table}

\begin{table}[H]
    \centering
    \caption{The dimension $8$ summands, which are all free modules. There are three families of them. In fact, these three families are the same as in Table~\ref{tab:421part-lem-3-dim8}, except with different indices for the first tensorand of every basis vector.}
    \label{tab:421part-lem-4-dim8}
    \scalebox{0.75}{
    \makebox[\textwidth][c]{
    \begin{tabular}{|c|c|c|}
    \hline
    Dimension $8$ Family $1$ summand: & Dimension $8$ Family $2$ summand: & Dimension $8$ Family $3$ summand:
    \\
\hline
    \end{tabular}
    }
    }
\end{table}

\begin{table}[H]
    \centering
    \caption{The dimension $12$ summands. They are the same as in Table~\ref{tab:421part-lem-3-dim12}, except with different indicies for the first tensorand of every basis vector and that there is an additional one of these summands. The pieces are arranged in the same way as Figure~\ref{fig:421part-lem-3-dim12}.}
    \label{tab:421part-lem-4-dim12}
    \scalebox{0.75}{
    \makebox[\textwidth][c]{
    %
    }
    }
\end{table}

\begin{table}[H]
    \centering
    \caption*{\textsc{Table~\ref{tab:421part-lem-4-dim12}} (continued).}
    \scalebox{0.75}{
    \makebox[\textwidth][c]{
    \begin{tabular}{|c||c|}
    \hline
    Dimension $12$ Family $1$ summand: & Dimension $12$ Family $2$ summand:
    \\
\hline
    \end{tabular}
    }
    }
\end{table}


\begin{proof}
Similar to the proof of Lemma~\ref{lem:41partition}, the explicit constructions show that the summands are both subrepresentations and indecomposable. The positions of the summands are summarized in Figure~\ref{fig:42-1-par-location}, where the black outline indicates the position of $V_i$ within the tensor product.

\begin{figure}
    \centering
    \caption{Positions of the indecomposable summands in $V_6 \otimes V$ (left) and $V_7 \otimes V$ (right).}
    \label{fig:42-1-par-location}
    \begin{subfigure}[H]{0.25\textwidth}
        \centering
        \begin{tikzpicture}[scale=0.5]
            \node at (8,0) {$1$};
            
            \node at (7,1) {$2$};
            \node at (8,1) {$2$};
            
            \node at (6,2) {$2$};
            \node at (7,2) {$5$};
            \node at (8,2) {$1$};
        
            \node at (5,3) {$1$};
            \node at (6,3) {$5$};
            \node at (7,3) {$6$};
        
            \node at (5,4) {$2$};
            \node at (6,4) {$8$};
            \node at (7,4) {$4$};
        
            \node at (5,5) {$5$};
            \node at (6,5) {$9$};
            \node at (7,5) {$1$};
        
            \node at (4,6) {$2$};
            \node at (5,6) {$8$};
            \node at (6,6) {$5$};
            
            \node at (3,7) {$1$};
            \node at (4,7) {$5$};
            \node at (5,7) {$8$};
            \node at (6,7) {$1$};
          
            \node at (3,8) {$2$};
            \node at (4,8) {$8$};
            \node at (5,8) {$5$};
            
            \node at (3,9) {$5$};
            \node at (4,9) {$9$};
            \node at (5,9) {$1$};
            
            \node at (2,10) {$2$};
            \node at (3,10) {$8$};
            \node at (4,10) {$5$}; 
        
            \node at (1,11) {$1$};
            \node at (2,11) {$5$};
            \node at (3,11) {$8$};  
            \node at (4,11) {$1$};
          
            \node at (1,12) {$2$};
            \node at (2,12) {$8$};
            \node at (3,12) {$5$};  
          
            \node at (1,13) {$4$};
            \node at (2,13) {$8$};
            \node at (3,13) {$1$};  
            
            \node at (1,14) {$4$};
            \node at (2,14) {$4$};
         
            \node at (1,15) {$3$};  
            \node at (2,15) {$1$};
            
            \node at (1,16) {$1$};
        
            \draw[semithick] (7.5,-0.5) -- (8.5,-0.5) -- (8.5,1.5) -- (7.5,1.5) -- (7.5,4.5) -- (6.5,4.5) -- (6.5,5.5) -- (5.5,5.5) -- (5.5,8.5) -- (4.5,8.5) -- (4.5,9.5) -- (3.5,9.5) -- (3.5,12.5) -- (2.5,12.5) -- (2.5,13.5) -- (1.5,13.5) -- (1.5,15.5) -- (0.5,15.5) -- (0.5,12.5) -- (1.5,12.5) -- (1.5,9.5) -- (2.5,9.5) -- (2.5,8.5) -- (3.5,8.5) -- (3.5,5.5) -- (4.5,5.5) -- (4.5,4.5) -- (5.5,4.5) -- (5.5,1.5) -- (6.5,1.5) -- (6.5,0.5) -- (7.5,0.5) -- (7.5,-0.5);
            
            \foreach \i in {0,2,4} {
                \node[draw, shape=diamond, semithick] at (6-\i,2+2*\i) {};
                \node[draw, shape=diamond, semithick] at (5-\i,3+2*\i) {};
                \node[draw, shape=diamond, semithick] at (5-\i,4+2*\i) {};
            }
        
            \foreach \i in {0,...,6}
                \draw (7-\i,1+2*\i) circle (0.5);
        
            \foreach \i in {0,2,4} {
                \draw[very thick] (7-\i,1+2*\i) -- (6-\i,3+2*\i);
                \draw[very thick, densely dotted] (6-\i,3+2*\i) -- (5-\i,5+2*\i);
            }
        \end{tikzpicture}
    \end{subfigure}
    \hfill
    \begin{subtable}[H]{0.45\textwidth}
        \centering
        \caption*{Key}
        \scalebox{0.9}{
        \begin{tabular}{|c|c|c|c|}
            \hline
            Dim. $8$ & Family 1 & Family 2 & Additional
            \\\hline
            \begin{tikzpicture}[scale=0.45]
                \node at (0,0) {$1$};
                \node at (0,1) {$1$};
                \node at (0,2) {$1$};
                \node at (0,3) {$1$};

                \node at (1,0) {$1$};
                \node at (1,1) {$1$};
                \node at (1,2) {$1$};
                \node at (1,3) {$1$};

                \node[draw, shape=diamond, semithick] at (0,0) {};
            \end{tikzpicture}
            &
            \begin{tikzpicture}[scale=0.45]
                \node at (2,0) {$1$};
                \node at (2,1) {$1$};
            
                \node at (1,0) {$1$};
                \node at (1,1) {$2$};
                \node at (1,2) {$2$};
                \node at (1,3) {$1$};
            
                \node at (0,2) {$1$};
                \node at (0,3) {$1$};
                \node at (0,4) {$1$};
                \node at (0,5) {$1$};

                \draw (1,0) circle (0.5);
                \draw (0,2) circle (0.5);
                \draw[very thick] (1,0) -- (0,2);
            \end{tikzpicture}
            &
            \begin{tikzpicture}[scale=0.45]
                \node at (2,0) {$1$};
                \node at (2,1) {$1$};
            
                \node at (1,0) {$1$};
                \node at (1,1) {$2$};
                \node at (1,2) {$2$};
                \node at (1,3) {$1$};
            
                \node at (0,2) {$1$};
                \node at (0,3) {$1$};
                \node at (0,4) {$1$};
                \node at (0,5) {$1$};

                \draw (1,0) circle (0.5);
                \draw (0,2) circle (0.5);
                \draw[very thick, densely dotted] (1,0) -- (0,2);
            \end{tikzpicture}
            &
            \begin{tikzpicture}[scale=0.45]
                \node at (2,0) {$1$};
                \node at (2,1) {$1$};
            
                \node at (1,0) {$1$};
                \node at (1,1) {$2$};
                \node at (1,2) {$2$};
                \node at (1,3) {$1$};
            
                \node at (0,2) {$1$};
                \node at (0,3) {$1$};
                \node at (0,4) {$1$};
                \node at (0,5) {$1$};

                \node[regular polygon, regular polygon sides=3, draw] at (1,0) {};
                \node[regular polygon, regular polygon sides=3, draw] at (0,2) {};
                \draw[very thick] (1,0) -- (0,2);
            \end{tikzpicture}
            \\\hline
        \end{tabular}
        }
    \end{subtable}
    \hfill
    \begin{subfigure}[H]{0.25\textwidth}
        \begin{tikzpicture}[scale=0.5]
            \node at (8,0) {$1$};
            
            \node at (7,1) {$2$};
            \node at (8,1) {$2$};
            
            \node at (6,2) {$2$};
            \node at (7,2) {$5$};
            \node at (8,2) {$1$};
        
            \node at (5,3) {$1$};
            \node at (6,3) {$5$};
            \node at (7,3) {$6$};
        
            \node at (5,4) {$2$};
            \node at (6,4) {$8$};
            \node at (7,4) {$4$};
        
            \node at (5,5) {$5$};
            \node at (6,5) {$9$};
            \node at (7,5) {$1$};
        
            \node at (4,6) {$2$};
            \node at (5,6) {$8$};
            \node at (6,6) {$5$};
            
            \node at (3,7) {$1$};
            \node at (4,7) {$5$};
            \node at (5,7) {$8$};
            \node at (6,7) {$1$};
          
            \node at (3,8) {$2$};
            \node at (4,8) {$8$};
            \node at (5,8) {$5$};
            
            \node at (3,9) {$5$};
            \node at (4,9) {$9$};
            \node at (5,9) {$1$};
            
            \node at (2,10) {$2$};
            \node at (3,10) {$8$};
            \node at (4,10) {$5$}; 
        
            \node at (1,11) {$1$};
            \node at (2,11) {$5$};
            \node at (3,11) {$8$};  
            \node at (4,11) {$1$};
          
            \node at (1,12) {$2$};
            \node at (2,12) {$8$};
            \node at (3,12) {$5$};  
          
            \node at (1,13) {$5$};
            \node at (2,13) {$9$};
            \node at (3,13) {$1$};  
            
            \node at (0,14) {$1$};
            \node at (1,14) {$7$};
            \node at (2,14) {$5$};  
            
            \node at (0,15) {$2$};
            \node at (1,15) {$6$};
            \node at (2,15) {$1$}; 
        
            \node at (0,16) {$3$};
            \node at (1,16) {$3$};
            
            \node at (0,17) {$2$}; 
        
            \node at (0,18) {$1$}; 
        
            \draw[semithick] (7.5,-0.5) -- (8.5,-0.5) -- (8.5,1.5) -- (7.5,1.5) -- (7.5,4.5) -- (6.5,4.5) -- (6.5,5.5) -- (5.5,5.5) -- (5.5,8.5) -- (4.5,8.5) -- (4.5,9.5) -- (3.5,9.5) -- (3.5,12.5) -- (2.5,12.5) -- (2.5,13.5) -- (1.5,13.5) -- (1.5,16.5) -- (0.5,16.5) -- (0.5,17.5) -- (-0.5,17.5) -- (-0.5,13.5) -- (0.5,13.5) -- (0.5,12.5) -- (1.5,12.5) -- (1.5,9.5) -- (2.5,9.5) -- (2.5,8.5) -- (3.5,8.5) -- (3.5,5.5) -- (4.5,5.5) -- (4.5,4.5) -- (5.5,4.5) -- (5.5,1.5) -- (6.5,1.5) -- (6.5,0.5) -- (7.5,0.5) -- (7.5,-0.5);
            
            \foreach \i in {0,2,4} {
                \node[draw, shape=diamond, semithick] at (6-\i,2+2*\i) {};
                \node[draw, shape=diamond, semithick] at (5-\i,3+2*\i) {};
                \node[draw, shape=diamond, semithick] at (5-\i,4+2*\i) {};
            }
        
            \foreach \i in {0,...,6}
                \draw (7-\i,1+2*\i) circle (0.5);
        
            \foreach \i in {0,2,4} {
                \draw[very thick] (7-\i,1+2*\i) -- (6-\i,3+2*\i);
                \draw[very thick, densely dotted] (6-\i,3+2*\i) -- (5-\i,5+2*\i);
            }

            \node[regular polygon, regular polygon sides=3, draw] at (1,13) {};
            \node[regular polygon, regular polygon sides=3, draw] at (0,15) {};
        
            \draw[very thick] (1,13) -- (0,15);
        \end{tikzpicture}
    \end{subfigure}
\end{figure}

Again, it can be checked that these summands span the respective tensor products. This is very similar to Lemma~\ref{lem:41partition}, and we have omitted the proof.
\end{proof}

This yet again verifies that Conjecture~\ref{conj:tensorpowers} holds for this monomial representation:

\begin{corollary}\label{4-2--1-poly}
    The function $P_V(n)$ is a quasi-polynomial with period $2$, given by $6n-1$ for odd $n$ and $6n+1$ for even $n$.
\end{corollary}

We close this section by noting that the computations involved in the theorem above resolve an analogue of a question asked by Benson and Symonds. We first recall some terminology. A $G$-module $M$ is called {\it Omega-algebraic} if there exist finitely many modules $M_1,..., M_m$ such that every non-projective indecomposable summand of $M^{\otimes n}$, for all $n$, is equal to $\Omega^{i}(M_j)$ for some $i \in \mathbb{Z}$ and $j=1,...,m$ (\cite[Definition~14.1]{bensontensorpower}). Omega-algebraic modules, and certain generalizations, have been used in analyzing the dimension of the largest non-projective summand of tensor powers \cite{CHU}. In~\cite{bensontensorpower}, Benson and Symonds posed the following question (and provided evidence that its answer was negative):

\begin{question}
    Must every faithful $G$-module be Omega-algebraic?
\end{question}

Lemma \ref{lem:421part} shows that the analogue of Benson and Symonds' question for graded $\alpha(r,s)$-modules is false.

\begin{lemma}\label{omega-alg}
The $(4,2)/(1)$ monomial representation is a faithful $G$-module which is not Omega-algebraic in the category of graded $\alpha(1,2)$-modules.
\end{lemma}

\begin{proof}
Let $V$ be the monomial representation corresponding to the partition $(4,2)/(1)$. Note that $V$ is faithful (as a $G$-module). It suffices to show that if $i \neq j$ and $i,j > 1$, then $V_i$ is not any syzygy $\Omega^k(V_j)$. This is equivalent to showing that if $i \neq j$, then $V_i$ is not any cosyzygy $\Omega^{-k}(V_j)$.
    
Since all projective modules are free modules in this case, the minimal projective cover of $V_j$ is a direct sum of many copies of the free module of rank $1$. By Lemma~\ref{lem:421part}, we know that for positive integer $j$, the graded representation $V_j$ has homogeneous components with degrees $(2j,j), (2j,j+1), (2j-1,j+1),$ and $(2j-1,j+4)$ that are dimension $1$ and homogeneous components with degrees $(2j-1,j+2)$ and $(2j-1,j+3)$ that are dimension $2$. By the same lemma, we know that there are no other homogeneous components with degree $(2j,n)$ for $n \neq j,j+1$ and no homogeneous components with degree $(2j-1,n)$ for $n \neq j+1,\dots, j+4$. Also, we know that there are no other homogeneous components with degree $(m,n)$ with $m > 2j$ or $n < j$. See Figure~\ref{fig:omegaalgebraicA} for the graded diagram.

This means that the minimal projective cover of $V_j$ has copies of the graded free module of rank 1 which has the homogeneous component in its top right corner as degrees $(2j,j+1),(2j-1,j+3),$ and $(2j-1,j+4)$. This implies that the cosyzygy of $V_j$ has homogeneous components with degrees $(2j-1,j-2),$ $(2j,j-2), (2j,j-1), (2j-1,j-1)$ that are dimension $1$ and homogeneous components with degrees $(2j-1,j)$ and $(2j-1,j+1)$ that are dimension $2$. Also, we know that there are no other homogeneous components with degree $(2j,n)$ for some other $n$, no other homogeneous components with degree $(2j-1,n)$ for some other $n$, no homogeneous components with degree $(m,n)$ for $m > 2j$ or $n < j-2$. See Figure~\ref{fig:omegaalgebraicB} for the graded diagram.

\begin{figure}[H]
    \centering
    \caption{The bottom right of the graded diagram of $V_j$ and cosyzygies, displaying the dimensions of the homogeneous components.}
    \begin{subfigure}[t]{0.4\textwidth}
        \centering
        \caption{The rightmost columns of $V_j$.}
        \label{fig:omegaalgebraicA}
        \begin{tikzpicture}[scale=0.5]
            \node at (0,0) {$1$};
            \node at (0,1) {$1$};
            \node at (-1,1) {$1$};
            \node at (-1,2) {$2$};
            \node at (-1,3) {$2$};
            \node at (-1,4) {$1$};
            \node at (-2,5) {$\ddots$};
        \end{tikzpicture}
    \end{subfigure}
    \begin{subfigure}[t]{0.4\textwidth}
        \centering
        \caption{The rightmost columns of $\Omega^{-1}(V_j)$.}
        \label{fig:omegaalgebraicB}
        \begin{tikzpicture}[scale=0.5]
            \node at (0,0) {$1$};
            \node at (1,0) {$1$};
            \node at (0,1) {$1$};
            \node at (1,1) {$1$};
            \node at (0,2) {$2$};
            \node at (0,3) {$2$};
            \node at (-1,4) {$\ddots$};
        \end{tikzpicture}
    \end{subfigure}
\end{figure}

The formation of the rightmost two columns of $\Omega^{-1}(V_j)$ is invariant to cosyzygies. Specifically, if $\Omega^{-k}(V_j)$ has its rightmost two columns with the relative locations and dimensions as in Figure~\ref{fig:omegaalgebraicB}, then let the degrees be $(m,n), (m+1,n), (m,n+1),$ and $(m+1,n+1)$ for the dimension $1$ homogeneous components and $(m,n+2)$ and $(m,n+3)$ for the dimension $2$ homogeneous components. Then, the minimum projective cover has copies of the rank $1$ free module, with the top right corners of these free modules having degrees $(m+1,n+1)$ and $(m,n+3)$, with the latter having multiplicity $2$. Then, the cosyzygy $\Omega^{-k-1}(V_j)$ has homogeneous components with degrees $(m,n-2), (m+1,n-2), (m,n-1),$ and $(m+1,n-1)$ that are dimension $1$ and has homogeneous components with degrees $(m,n)$ and $(m,n+1)$ that are dimension $2$. This is the exact same relative locations and dimensions as the homogeneous components shown in Figure~\ref{fig:omegaalgebraicB}, so this is invariant.

Hence, all $\Omega^{-k}(V_j)$ for $k > 0$ have rightmost columns as in Figure~\ref{fig:omegaalgebraicB}. However, none of the $V_i$ have these rightmost columns by Lemma~\ref{lem:421part}. Therefore, $V_i \not\cong \Omega^{-k}(V_j)$ for $k > 0$, as desired.
\end{proof}

\subsection{$(6,1)$ Monomial representation}

Let $V$ be the monomial representation of $\mathbb{Z}/2\mathbb{Z} \times \mathbb{Z}/8\mathbb{Z}$ (or $\alpha(1,3)$) corresponding to the partition $(6,1)$. The quasi-polynomial is $P_V(n) = 2n^2 + 4n + 1$, a quadratic. We have the following decomposition:

\begin{proposition}
We have the following decomposition into indecomposable summand:
$$V_{2k-1} \otimes V = V_{2k} \oplus \bigoplus_{i=1}^k W^{A}_{16i} \oplus \bigoplus_{i=1}^{k-1} W^{B}_{16i} \oplus (2k^2-k-1)F \oplus (k+1)W_8,$$
$$V_{2k} \otimes V = V_{2k+1} \oplus \bigoplus_{i=1}^{k} W^{A}_{16i+8} \oplus \bigoplus_{i=1}^{k} W^{B}_{16i-8} \oplus (2k^2+k)F \oplus (k)W_8,$$
where $\dim W^{A}_i = i, \dim W^{B}_i = i, \dim W_8 = 8$, and $F$ is a free module of dimension $16$.
\end{proposition}

We explicitly write the decomposition with a long but stronger lemma.

\begin{lemma}
We claim the following.
\begin{enumerate}[leftmargin=*]
    \item The representation $V_{2k+1}$ for $k\geq 1$ is given by the graded monomial diagram in Table~\ref{tab:61part-lem-1}.

    \item The representation $V_{2k}$ for $k \geq 1$ is given by the graded monomial diagram in Table~\ref{tab:61part-lem-2}.
    
    \item The representation $V_{2k-1} \otimes V$ decomposes into the following summands: $V_{2k}$ (given in Table~\ref{tab:61part-lem-3-V2k}); summands $W_{i}^{A}$ (given in Table~\ref{tab:61part-lem-3-WA}); summands $W_{i}^{B}$ (given in Table~\ref{tab:61part-lem-3-WB}); dimension $16$ summands (given in Table~\ref{tab:61part-lem-3-16}); and dimension $8$ summands (given in Table~\ref{tab:61part-lem-3-8}).

    \item The representation $V_{2k}\otimes V$ decomposes into the following summands: $V_{2k+1}$ (given in Table~\ref{tab:61part-lem-4-V2k+1}); summands $W_{i}^{A}$ (given in Table~\ref{tab:61part-lem-4-WA}); summands $W_{i}^{B}$ (given in Table~\ref{tab:61part-lem-4-WB}); dimension $16$ summands (given in Table~\ref{tab:61part-lem-4-16}); and dimension $8$ summands (given in Table~\ref{tab:61part-lem-4-8}).
\end{enumerate}
\end{lemma}

\begin{table}[H]
    \centering
    \caption{The pieces of $V_{2k+1}$. These pieces show the nonzero actions of $x$ and $y$, and they overlap with each other. They are arranged as in Figure~\ref{fig:61part-lem-1} for $V_5$.}
    \label{tab:61part-lem-1}
    \scalebox{0.75}{
    \makebox[\textwidth][c]{
\begin{tabular}{|c|c|}
        \hline
        $(A,1)$ piece: & $(B,h)$ piece:
        \\
\hline
    \end{tabular}
    }
    }
\end{table}

\begin{table}[H]
    \centering
    \caption*{\textsc{Table~\ref{tab:61part-lem-1}} (continued).}
    \scalebox{0.75}{
    \makebox[\textwidth][c]{
    \begin{tabular}{|>{\centering\arraybackslash}p{24cm}|}
        \hline
        $(A,h)$ piece:
        \\
\hline
    \end{tabular}
    }
    }
\end{table}

\begin{table}[H]
    \centering
    \caption{The pieces of $V_{2k}$. These pieces show the nonzero actions of $x$ and $y$, and they overlap each other. They are arranged as in Figure~\ref{fig:61part-lem-2} for $V_6$.}
    \label{tab:61part-lem-2}
    \scalebox{0.75}{
    \makebox[\textwidth][c]{
    \begin{tabular}{|c|c|}
        \hline
        $(A,1)$ piece: & $(B,h)$ piece:
        \\
\hline
    \end{tabular}
    }
    }
\end{table}

\begin{table}[H]
    \centering
    \caption*{\textsc{Table~\ref{tab:61part-lem-2}} (continued).}
    \scalebox{0.75}{
    \makebox[\textwidth][c]{
    \begin{tabular}{|>{\centering\arraybackslash}p{24cm}|}
        \hline
        $(A,h)$ piece:
        \\
\hline
    \end{tabular}
    }
    }
\end{table}

\begin{figure}[H]
    \centering
    \caption{The location of the pieces for $V_5$ and $V_6$. The outlines show the $(A,i)$ pieces, and the shaded regions show the $(B,i)$ pieces.}
    \label{fig:61part-lem-1,2}
    \begin{subfigure}[t]{0.35\textwidth}
        \centering
        \caption{$V_5$}
        \label{fig:61part-lem-1}
        %
    \end{subfigure}
    \hfill
    \begin{subfigure}[t]{0.35\textwidth}
        \centering
        \caption{$V_6$}
        \label{fig:61part-lem-2}
        \scalebox{0.95}{
        %
        }
    \end{subfigure}
\end{figure}

\begin{figure}[H]
    \centering
    \caption{The relation between the graded monomial diagrams dimensions of $V_i$ and $V_{i-1}$. Figure~\ref{fig:61part-lem-1,2} gives two examples: $V_5$ and $V_6$.}
    %
\end{figure}

\begin{table}[H]
    \centering
    \caption{The pieces of $V_{2k}$, where $V_{i}(X,\ell)$ denotes the $(X,\ell)$ piece of $V_i$.}
    \label{tab:61part-lem-3-V2k}
    \scalebox{0.75}{
    \makebox[\textwidth][c]{
    \begin{tabular}{|c|c|c|}
        \hline
        $(A,1)$ piece: & $(A,h)$ piece: & $(A,k)$ piece:
        \\
\hline
    \end{tabular}
    }
    }
\end{table}

\begin{table}[H]
    \centering
    \caption{The summands $W_{16(k-i)}^{A}$, each of which is written in two pieces. These pieces show the nonzero actions of $x$ and $y$ and overlap each other. Note that $i=0$ is a special edge case, in which we define $v_{i,j}^{A,0} := v_{i,j}^{A,1}$ and $v_{i,j}^{B,0} := 0$.}
    \label{tab:61part-lem-3-WA}
    \scalebox{0.8}{
    \makebox[\textwidth][c]{
    \begin{tabular}{|c|c|}
        \hline
        \multicolumn{2}{|c|}{
        Summand $W_{16(k-i)}^{A}$:
        }
        \\
\hline
    \end{tabular}
    }
    }
\end{table}

\begin{table}[H]
    \centering
    \caption{The summands $W_{16(k-i)}^{B}$. }
    \label{tab:61part-lem-3-WB}
    \scalebox{0.8}{
    \makebox[\textwidth][c]{
    \begin{tabular}{|c|}
        \hline
        Summand $W_{16(k-i)}^{B}$:
        \\
\hline
    \end{tabular}
    }
    }
\end{table}

\begin{table}[H]
    \centering
    \caption{The dimension $16$ summands, which are all free modules. There are seven families of them, which differ in only how they are expressed below. Note that for space reasons, we denote $v_{i,j}^{(A,i)+(A,i+1)} := v_{i,j}^{A,i} + v_{i,j}^{A,i+1}$.}
    \label{tab:61part-lem-3-16}
    \scalebox{0.8}{
    \makebox[\textwidth][c]{
    \begin{tabular}{|c|c|}
        \hline
        Dimension $16$ Family $1A$ summand: & Dimension $16$ Family $1B$ summand:
        \\
\hline
    \end{tabular}
    }
    }
\end{table}

\begin{table}[H]
    \centering
    \caption*{\textsc{Table~\ref{tab:61part-lem-3-16}} (continued).}
    \scalebox{0.75}{
    \makebox[\textwidth][c]{
    \begin{tabular}{|c|c|c|}
        \hline
        Dimension $16$ Family $2A$ summand: & Dimension $16$ Family $2B$ summand: & Dimension $16$ Family $2C$ summand:
        \\
\hline
    \end{tabular}
    }
    }
\end{table}

\begin{table}[H]
    \centering
    \caption*{\textsc{Table~\ref{tab:61part-lem-3-16}} (continued).}
    \scalebox{0.8}{
    \makebox[\textwidth][c]{
    \begin{tabular}{|c|c|}
        \hline
        Dimension $16$ Family $3A$ summand: & Dimension $16$ Family $3B$ summand:
        \\
\hline
    \end{tabular}
    }
    }
\end{table}

\begin{table}[H]
    \centering
    \caption{The dimension $8$ summands. There are three families of them, which differ in only how they are expressed below.}
    \label{tab:61part-lem-3-8}
    \scalebox{0.85}{
    \makebox[\textwidth][c]{
    \begin{tabular}{|c|c|c|}
        \hline
        Dimension $8$ Family $1$ summand: & Dimension $8$ Family $2$ summand: & Dimension $8$ Family $3$ summand:
        \\
\hline
    \end{tabular}
    }
    }
\end{table}

\begin{table}[H]
    \centering
    \caption{The pieces of $V_{2k+1}$, where $V_{i}(X,\ell)$ denotes the $(X,\ell)$ piece of $V_i$.}
    \label{tab:61part-lem-4-V2k+1}
    \scalebox{0.75}{
    \makebox[\textwidth][c]{
    \begin{tabular}{|c|c|c|}
        \hline
        $(A,1)$ piece: & $(A,h)$ piece: & $(A,k)$ piece:
        \\
\hline
    \end{tabular}
    }
    }
\end{table}

\begin{table}[H]
    \centering
    \caption{The summands $W_{16(k-i)+8}^{A}$, each of which is written in two pieces. These pieces show the nonzero actions of $x$ and $y$ and overlap each other. Note that $i=0$ is a special edge case, in which we define $v_{i,j}^{A,0} := v_{i,j}^{A,1}$ and $v_{i,j}^{B,0} := 0$.
    }
    \label{tab:61part-lem-4-WA}
    \scalebox{0.8}{
    \makebox[\textwidth][c]{
    \begin{tabular}{|c|c|}
        \hline
        \multicolumn{2}{|c|}{
        Summand $W_{16(k-i)+8}^{A}$:
        }
        \\
\hline
    \end{tabular}
    }
    }
\end{table}

\begin{table}[H]
    \centering
    \caption{The summands $W_{16(k-i)+8}^{B}$. Note that when $i=k$, some of the vectors do not exist. This construction works with those nonexistent vectors removed.
    }
    \label{tab:61part-lem-4-WB}
    \scalebox{0.8}{
    \makebox[\textwidth][c]{
    \begin{tabular}{|c|}
        \hline
        Summand $W_{16(k-i)+8}^{B}$:
        \\
\hline
    \end{tabular}
    }
    }
\end{table}

\begin{table}[H]
    \centering
    \caption{The dimension $16$ summands, which are all free modules. There are seven families of them, which differ in only how they are expressed below. Note that for space reasons, we denote $v_{i,j}^{(A,i)+(A,i+1)} := v_{i,j}^{A,i} + v_{i,j}^{A,i+1}$.
    }
    \label{tab:61part-lem-4-16}
    \scalebox{0.8}{
    \makebox[\textwidth][c]{
    \begin{tabular}{|c|c|}
        \hline
        Dimension $16$ Family $1A$ summand: & Dimension $16$ Family $1B$ summand:
        \\
\hline
    \end{tabular}
    }
    }
\end{table}

\begin{table}[H]
    \centering
    \caption*{\textsc{Table~\ref{tab:61part-lem-4-16}} (continued).
    }
    \scalebox{0.75}{
    \makebox[\textwidth][c]{
    \begin{tabular}{|c|c|c|}
        \hline
        Dimension $16$ Family $2A$ summand: & Dimension $16$ Family $2B$ summand: & Dimension $16$ Family $2C$ summand:
        \\
\hline
    \end{tabular}
    }
    }
\end{table}

\begin{table}[H]
    \centering
    \caption*{\textsc{Table~\ref{tab:61part-lem-4-16}} (continued).
    }
    \scalebox{0.8}{
    \makebox[\textwidth][c]{
    \begin{tabular}{|c|c|}
        \hline
        Dimension $16$ Family $3A$ summand: & Dimension $16$ Family $3B$ summand:
        \\
\hline
    \end{tabular}
    }
    }
\end{table}

\begin{table}[H]
    \centering
    \caption{The dimension $8$ summands. There are two families of them, which differ in only how they are expressed below.}
    \label{tab:61part-lem-4-8}
    \scalebox{0.85}{
    \makebox[\textwidth][c]{
    \begin{tabular}{|c|c|}
        \hline
        Dimension $8$ Family $1$ summand: & Dimension $8$ Family $2$ summand:
        \\
\hline
    \end{tabular}
    }
    }
\end{table}

The proof is similar to the previous cases.

\begin{proof}
Similar to the proof of Lemma~\ref{lem:41partition}, the explicit constructions show that the summands are both subrepresentations and indecomposable. The positions of the summands are summarized in Figure~\ref{fig:61-par-location}, where the black outline indicates the position of $V_i$ within the tensor product.

We only need to check that the summands span the tensor product. We proceed by induction. The base cases can be manually checked.

Assume the claimed decomposition is true for $V_{2k-2} \otimes V$. Consider the tensor product $V_{2k-1}\otimes V$. The homogeneous components with degrees $(i,*)$ for $i=2k+2, \dots, 4k$ correspond to the homogeneous components in $V_{2k-2}\otimes V$ with degrees $(j,*)$ for $j=2k,\dots,4k-2$, since the summands in item (3) of the lemma correspond to the summands in item (4) of the lemma, with index shifts. This means that we only need to check the components in $V_{2k-1}\otimes V$ with degree $(i,*)$ for $i=2k,2k+1$ (i.e.,~the leftmost two columns). This is very similar to Lemma~\ref{lem:41partition}, except that the columns, from right to left, grow larger (hence the quadratic in Corollary~\ref{3-1-1-poly}). However, using the lemma, it can be seen that the additional homogeneous components that need to be checked for these growing columns are the same computations as previous ones, simplifying the number of cases. With this, in the same way as Lemma~\ref{lem:41partition}, it can be checked that the summands span the tensor product, and we have omitted the proof.
\end{proof}

As a corollary, we observe that Conjecture~\ref{conj:tensorpowers} again holds in this case; this time, the function $P_V(n)$ is a quadratic.

\begin{corollary}\label{6-1-poly}
The function $P_V(n)$ is the quadratic polynomial $2n^2+4n+1$.
\end{corollary}

With this decomposition, we can intuitively see why $P_V(n)$ is a quadratic in this case. In particular, the $(A,i)$ pieces of some $V_{n}$ are each increasing by a constant amount (in this case, by $8$) for each successive tensor power, since we are adding another column to the left of each of the pieces. The number of pieces $(A,i)$ is linear in relation to $n$, so the dimension of the $V_n$ increases by an amount linear to $n$ for each successive tensor power. 

\begin{figure}
    \centering
    \caption{Positions of the indecomposable summands in $V_4 \otimes V$ (left) and $V_5 \otimes V$ (right).}
    \label{fig:61-par-location}
    \begin{subfigure}[H]{0.25\textwidth}
        \centering
        \scalebox{0.9}{
        \begin{tikzpicture}[scale=0.5]
            \node at (5,0) {$1$};
            \node at (4,0) {$2$};
            \node at (3,0) {$1$};
            \node at (4,1) {$2$};
            \node at (3,1) {$3$};
            \node at (2,1) {$1$};
            \node at (4,2) {$2$};
            \node at (3,2) {$4$};
            \node at (2,2) {$2$};
            \node at (4,3) {$2$};
            \node at (3,3) {$5$};
            \node at (2,3) {$3$};
            \node at (4,4) {$2$};
            \node at (3,4) {$6$};
            \node at (2,4) {$4$};
            \node at (4,5) {$2$};
            \node at (3,5) {$8$};
            \node at (2,5) {$7$};
            \node at (1,5) {$1$};
            \node at (3,6) {$7$};
            \node at (2,6) {$10$};
            \node at (1,6) {$3$};
            \node at (3,7) {$5$};
            \node at (2,7) {$10$};
            \node at (1,7) {$6$};
            \node at (0,7) {$1$};
            \node at (3,8) {$4$};
            \node at (2,8) {$10$};
            \node at (1,8) {$8$};
            \node at (0,8) {$2$};
            \node at (3,9) {$2$};
            \node at (2,9) {$9$};
            \node at (1,9) {$11$};
            \node at (0,9) {$4$};
            \node at (3,10) {$1$};
            \node at (2,10) {$8$};
            \node at (1,10) {$13$};
            \node at (0,10) {$6$};
            \node at (2,11) {$6$};
            \node at (1,11) {$15$};
            \node at (0,11) {$9$};
            \node at (2,12) {$3$};
            \node at (1,12) {$14$};
            \node at (0,12) {$12$};
            \node at (2,13) {$1$};
            \node at (1,13) {$11$};
            \node at (0,13) {$13$};
            \node at (1,14) {$9$};
            \node at (0,14) {$14$};
            \node at (1,15) {$6$};
            \node at (0,15) {$13$};
            \node at (1,16) {$4$};
            \node at (0,16) {$12$};
            \node at (1,17) {$1$};
            \node at (0,17) {$9$};
            \node at (0,18) {$6$};
            \node at (0,19) {$4$};
            \node at (0,20) {$2$};
            \node at (0,21) {$1$};

            \draw (3.5,-0.5) -- (5.5,-0.5) -- (5.5,0.5) -- (4.5,0.5) -- (4.5,5.5) -- (3.5,5.5) -- (3.5,8.5) -- (2.5,8.5) -- (2.5,12.5) -- (1.5,12.5) -- (1.5,16.5) -- (0.5,16.5) -- (0.5,20.5) -- (-0.5,20.5) -- (-0.5,10.5) -- (0.5,10.5) -- (0.5,6.5) -- (1.5,6.5) -- (1.5,4.5) -- (2.5,4.5) -- (2.5,0.5) -- (3.5,0.5) -- (3.5,-0.5);

            \draw[semithick] (0-0.5,12-0.5) rectangle node{} ++(1,1);
            \draw[semithick] (0-0.5,14-0.5) rectangle node{} ++(1,1);

            \draw (3,0) circle (0.5);
            \draw (1,6) circle (0.5);

            \node[regular polygon, regular polygon sides=3, draw] at (2,5) {};
            \node[regular polygon, regular polygon sides=3, draw] at (0,11) {};
            
            \node[draw, shape=diamond, semithick] at (2,1) {};
            \node[draw, shape=diamond, semithick] at (2,2) {};
            \node[draw, shape=diamond, semithick] at (2,3) {};

            \node[draw, shape=diamond, semithick] at (1,5) {};
            \node[draw, shape=diamond, semithick] at (1,6) {};

            \node[draw, shape=diamond, semithick] at (0,7) {};
            \node[draw, shape=diamond, semithick] at (0,8) {};
            \node[draw, shape=diamond, minimum size = 0.4cm, semithick] at (0,9) {};
            \node[draw, shape=diamond, minimum size = 0.6cm, semithick] at (0,9) {};
            \node[draw, shape=diamond, semithick] at (0,10) {};
        \end{tikzpicture}
        }
    \end{subfigure}
    \hfill
    \begin{subtable}[H]{0.45\textwidth}
        \centering
        \caption*{Key}
        \scalebox{0.9}{
        \begin{tabular}{|c|c|c|c|}
            \hline
            $W_i^A$ & $W_i^B$ & Dim.~$16$ & Dim.~$8$
            \\\hline
            \begin{tikzpicture}[scale=0.4]
                \node at (2,0) {$1$};
                \node at (2,1) {$1$};
                \node at (2,2) {$1$};
                \node at (2,3) {$1$};
                \node at (2,4) {$1$};
                \node at (2,5) {$1$};

                \node at (1,0) {$1$};
                \node at (1,1) {$1$};
                \node at (1,2) {$1$};
                \node at (1,3) {$1$};
                \node at (1,4) {$2$};
                \node at (1,5) {$2$};
                \node at (1,6) {$1$};
                \node at (1,7) {$1$};

                \node at (0,4) {$1$};
                \node at (0,5) {$1$};
                \node at (0,6) {$1$};
                \node at (0,7) {$1$};
                \node at (0,8) {$1$};
                \node at (0,9) {$1$};
                \node at (0,10) {$1$};
                \node at (0,11) {$1$};

                \node at (-1,11) {$\ddots$};

                \node at (-2,8) {$1$};
                \node at (-2,9) {$1$};
                \node at (-2,10) {$1$};
                \node at (-2,11) {$1$};
                \node at (-2,12) {$1$};
                \node at (-2,13) {$1$};
                \node at (-2,14) {$1$};
                \node at (-2,15) {$1$};

                \draw (1,0) circle (0.5);
            \end{tikzpicture}
            &
            \begin{tikzpicture}[scale=0.4]
                \node at (2,0) {$1$};
                \node at (2,1) {$1$};

                \node at (1,0) {$1$};
                \node at (1,1) {$1$};
                \node at (1,2) {$1$};
                \node at (1,3) {$1$};
                \node at (1,4) {$1$};
                \node at (1,5) {$1$};

                \node at (0,2) {$1$};
                \node at (0,3) {$1$};
                \node at (0,4) {$1$};
                \node at (0,5) {$1$};
                \node at (0,6) {$1$};
                \node at (0,7) {$1$};
                \node at (0,8) {$1$};
                \node at (0,9) {$1$};

                \node at (-1,9) {$\ddots$};

                \node at (-2,6) {$1$};
                \node at (-2,7) {$1$};
                \node at (-2,8) {$1$};
                \node at (-2,9) {$1$};
                \node at (-2,10) {$1$};
                \node at (-2,11) {$1$};
                \node at (-2,12) {$1$};
                \node at (-2,13) {$1$};
                
                \node[regular polygon, regular polygon sides=3, draw] at (1,0) {};
            \end{tikzpicture}
            &
            \begin{tikzpicture}[scale=0.4]
                \node at (0,0) {$1$};
                \node at (0,1) {$1$};
                \node at (0,2) {$1$};
                \node at (0,3) {$1$};
                \node at (0,4) {$1$};
                \node at (0,5) {$1$};
                \node at (0,6) {$1$};
                \node at (0,7) {$1$};
                \node at (1,0) {$1$};
                \node at (1,1) {$1$};
                \node at (1,2) {$1$};
                \node at (1,3) {$1$};
                \node at (1,4) {$1$};
                \node at (1,5) {$1$};
                \node at (1,6) {$1$};
                \node at (1,7) {$1$};
            
                \node[draw, shape=diamond, semithick] (0,0) {};
            \end{tikzpicture}
            &
            \begin{tikzpicture}[scale=0.4]
                \draw[semithick] (-0.5,-0.5) rectangle node{} ++(1,1);
            
                \node at (0,0) {$1$};
                \node at (0,1) {$1$};
                \node at (0,2) {$1$};
                \node at (0,3) {$1$};
                \node at (0,4) {$1$};
                \node at (0,5) {$1$};
                \node at (0,6) {$1$};
                \node at (0,7) {$1$};
            \end{tikzpicture}
            \\\hline
            \end{tabular}
        }
    \end{subtable}
    \hfill
    \begin{subfigure}[H]{0.25\textwidth}
        \centering
        \scalebox{0.9}{
        \begin{tikzpicture}[scale=0.5]
            \node at (5,0) {$1$};
            \node at (4,0) {$2$};
            \node at (3,0) {$1$};
            \node at (4,1) {$2$};
            \node at (3,1) {$3$};
            \node at (2,1) {$1$};
            \node at (4,2) {$2$};
            \node at (3,2) {$4$};
            \node at (2,2) {$2$};
            \node at (4,3) {$2$};
            \node at (3,3) {$5$};
            \node at (2,3) {$3$};
            \node at (4,4) {$2$};
            \node at (3,4) {$6$};
            \node at (2,4) {$4$};
            \node at (4,5) {$2$};
            \node at (3,5) {$8$};
            \node at (2,5) {$7$};
            \node at (1,5) {$1$};
            \node at (3,6) {$7$};
            \node at (2,6) {$10$};
            \node at (1,6) {$3$};
            \node at (3,7) {$5$};
            \node at (2,7) {$10$};
            \node at (1,7) {$6$};
            \node at (3,8) {$4$};
            \node at (2,8) {$10$};
            \node at (1,8) {$8$};
            \node at (3,9) {$2$};
            \node at (2,9) {$9$};
            \node at (1,9) {$11$};
            \node at (3,10) {$1$};
            \node at (2,10) {$8$};
            \node at (1,10) {$13$};
            \node at (2,11) {$6$};
            \node at (1,11) {$15$};
            \node at (2,12) {$3$};
            \node at (1,12) {$14$};
            \node at (2,13) {$1$};
            \node at (1,13) {$11$};
            \node at (1,14) {$9$};
            \node at (1,15) {$6$};
            \node at (1,16) {$4$};
            \node at (1,17) {$1$};
            
            \node at (0,7) {$1$};
            \node at (0,8) {$2$};
            \node at (0,9) {$4$};
            \node at (0,10) {$6$};
            \node at (0,11) {$10$};
            \node at (0,12) {$14$};
            \node at (0,13) {$16$};
            \node at (0,14) {$17$};
            \node at (0,15) {$16$};
            \node at (0,16) {$15$};
            \node at (0,17) {$12$};
            \node at (0,18) {$8$};
            \node at (0,19) {$5$};
            \node at (0,20) {$3$};
            \node at (0,21) {$1$};
            \node at (-1,11) {$1$};
            \node at (-1,12) {$3$};
            \node at (-1,13) {$6$};
            \node at (-1,14) {$9$};
            \node at (-1,15) {$12$};
            \node at (-1,16) {$15$};
            \node at (-1,17) {$17$};
            \node at (-1,18) {$17$};
            \node at (-1,19) {$15$};
            \node at (-1,20) {$13$};
            \node at (-1,21) {$10$};
            \node at (-1,22) {$7$};
            \node at (-1,23) {$4$};
            \node at (-1,24) {$2$};
            \node at (-1,25) {$1$};

            \draw (3.5,-0.5) -- (5.5,-0.5) -- (5.5,0.5) -- (4.5,0.5) -- (4.5,5.5) -- (3.5,5.5) -- (3.5,8.5) -- (2.5,8.5) -- (2.5,12.5) -- (1.5,12.5) -- (1.5,16.5) -- (0.5,16.5) -- (0.5,20.5) -- (-0.5,20.5) -- (-0.5,24.5) -- (-1.5,24.5) -- (-1.5,12.5) -- (-0.5,12.5) -- (-0.5,10.5) -- (0.5,10.5) -- (0.5,6.5) -- (1.5,6.5) -- (1.5,4.5) -- (2.5,4.5) -- (2.5,0.5) -- (3.5,0.5) -- (3.5,-0.5);

            \draw[semithick] (-1-0.5,14-0.5) rectangle node{} ++(1,1);
            \draw[semithick] (-1-0.5,15-0.5) rectangle node{} ++(1,1);
            \draw[semithick] (-1-0.5,16-0.5) rectangle node{} ++(1,1);
            \draw[semithick] (-1-0.5,18-0.5) rectangle node{} ++(1,1);

            \draw (3,0) circle (0.5);
            \draw (1,6) circle (0.5);
            \draw (-1,12) circle (0.5);

            \node[regular polygon, regular polygon sides=3, draw] at (2,5) {};
            \node[regular polygon, regular polygon sides=3, draw] at (0,11) {};
            
            \node[draw, shape=diamond, semithick] at (2,1) {};
            \node[draw, shape=diamond, semithick] at (2,2) {};
            \node[draw, shape=diamond, semithick] at (2,3) {};

            \node[draw, shape=diamond, semithick] at (1,5) {};
            \node[draw, shape=diamond, semithick] at (1,6) {};

            \node[draw, shape=diamond, semithick] at (0,7) {};
            \node[draw, shape=diamond, semithick] at (0,8) {};
            \node[draw, shape=diamond, minimum size = 0.4cm, semithick] at (0,9) {};
            \node[draw, shape=diamond, minimum size = 0.6cm, semithick] at (0,9) {};
            \node[draw, shape=diamond, semithick] at (0,10) {};

            \node[draw, shape=diamond, semithick] at (-1,11) {};
            \node[draw, shape=diamond, semithick] at (-1,12) {};
            \node[draw, shape=diamond, semithick] at (-1,13) {};
            \node[draw, shape=diamond, semithick] at (-1,14) {};
            
        \end{tikzpicture}
        }
    \end{subfigure}
\end{figure}

\section{Further questions and conjectures}\label{sec:otherconjs}

The data on tensor powers computed with Magma prompts the following questions, which we observe hold in many cases:

\begin{question}
If $V$ is an odd-dimensional indecomposable monomial representation, then we have that $(\dim V)^n \equiv P_V(n) \pmod 4$. In particular, $\dim V_{2k} \equiv 1 \pmod 4$ and $\dim V_{2k+1} \equiv \dim V \pmod 4$.
\end{question}

\begin{question}
Let $V$ be an odd-dimensional indecomposable monomial representation. Then, $V^{\otimes n}$ is the direct sum of an odd-dimensional indecomposable representation (with dimension $P_V(n)$) and indecomposable representations with dimensions divisible by $4$.
\end{question}

Table~\ref{tab:tenspowdata} summarizes the monomial modules that we have computational evidence for the polynomial or quasi-polynomial (computational evidence defined by having at least $n+2$ data points for a degree $n$ polynomial guess). The notation $[f(x), g(x)]$ is defined as a quasi-polynomial that is $f(x)$ when $x$ is odd and $g(x)$ when $x$ is even. Note that all quasi-polynomials in the table are either of period $1$ or $2$. 

The monomial representations are not listed in any important order. The ones not listed are either proved previously ($180^\circ$-symmetric; $(2^m,1)$; staircase; $(3,1,1)$; $(4,2)/(1)$; or $(6,1)$ partitions) or have been difficult to compute sufficient data for. Also, note that some monomial representations are $V_i$'s for smaller monomial partitions. For example, the $(5,4,1)/(1)$ monomial representation is $V_2$ for the monomial representation $V := (4,1)$, so if $P_V(n) = 4n+1$ then $P_{V_2}(n) = 8n+1$, as supported by computation in the table below.

\newpage

\begin{longtable}{|c|c|}
\caption{Monomial representations and conjectured polynomials from computational evidence.}
\label{tab:tenspowdata}
\vspace{-3mm}
\\
\hline
Partition & Computed Quasi-polynomial
\\\hline
$(3,2)$ & $[10x-5,6x+1]$
\\\hline
$(5,1,1)$ & $[18x-11,10x+1]$
\\\hline
$(4,3)$ & $[4x+3, 4x+1]$
\\\hline
$(4,1,1,1)$ & $[8x-1,8x+1]$
\\\hline
$(6,2)/(1)$ & $[10x-3,10x+1]$
\\\hline
$(5,3)/(1)$ & $[12x-5,12x-7]$
\\\hline
$(5,2,1)/(1)$ & $6x+1$
\\\hline
$(4,3,1)/(1)$ & $[12x-4,12x+1]$
\\\hline
$(4,2,1,1)/(1)$ & $[20x-13,12x+1]$
\\\hline
$(6,3)/(2)$ & $[4x+3,8x+1]$
\\\hline
$(5,3,1)/(2)$ & $[8x-1,12x+1]$
\\\hline
$(5,2,2)/(1,1)$ & $[10x-3,10x+1]$
\\\hline
$(4,4,1)/(2)$ & $[12x-5,12x-7]$
\\\hline
$(4,3,2)/(2)$ & $[14x-7,10x+1]$
\\\hline
$(4,3,2)/(1,1)$ & $38x-31$
\\\hline
$(5,4,1)/(3)$ & $2x^2+4x+1$
\\\hline
$(5,3,2)/(2,1)$ & $[10x-3,10x+1]$
\\\hline
$(7,1,1)$ & $[20x-11,12x+1]$
\\\hline
$(5,4)$ & $[20x-11,12x+1]$
\\\hline
$(4,4,1)$ & $[12x-3,12x+1]$
\\\hline
$(8,2)/(1)$ & $[12x-3,12x+1]$
\\\hline
$(7,3)/(1)$ & $12x^2-4x+1$
\\\hline
$(7,2,1)/(1)$ & $2x^2+6x+1$
\\\hline
$(6,4)/(1)$ & $48x-39$
\\\hline
$(5,4,1)/(1)$ & $8x+1$
\\\hline
$(5,2,1,1,1)/(1)$ & $8x+1$
\\\hline
$(4,4,1,1)/(1)$ & $[12x-3,12x+1]$
\\\hline
$(7,3,1)/(2)$ & $[18x-9,18x+1]$
\\\hline
$(6,4,1)/(2)$ & $16x-7$
\\\hline
$(6,3,2)/(2)$ & $[14x-5,10x+1]$
\\\hline
$(6,3,1,1)/(2)$ & $[18x-9,14x+1]$
\\\hline
$(5,5,1)/(2)$ & $8x^2+1$
\\\hline
$(5,4,2)/(2)$ & $16x-7$
\\\hline
$(5,4,2)/(1,1)$ & $8x^2+1$
\\\hline
$(5,3,1,1,1)/(2)$ & $[10x-1,18x+1]$
\\\hline
$(8,4)/(3)$ & $[12x-3,12x+1]$
\\\hline
$(7,4,1)/(3)$ & $[8x+1,12x+1]$
\\\hline
$(7,3,2)/(2,1)$ & $[16x-7,16x+1]$
\\\hline
$(6,3,2,1)/(2,1)$ & $8x+1$
\\\hline
$(5,5,2)/(2,1)$ & $8x+1$
\\\hline
$(5,4,3)/(3)$ & $8x+1$
\\\hline
$(4,4,3,1)/(3)$ & $56x-47$
\\\hline
$(7,5,1)/(4)$ & $[12x-3, 16x+1]$
\\\hline
$(7,4,2)/(3,1)$ & $[10x-1,10x+1]$
\\\hline
$(6,4,2,1)/(3,1)$ & $[10x-1,18x+1]$
\\\hline
$(6,3,3,1)/(2,2)$ & $64x-55$
\\\hline
$(5,5,2,1)/(3,1)$ & $[14x-5, 14x-7]$
\\\hline
$(5,5,1,1,1)/(4)$ & $[16x-,16x+1]$
\\\hline
$(5,3,3,2)/(2,2)$ & $[14x-5,10x+1]$
\\\hline
$(7,6,1)/(5)$ & $4x^2+4x+1$
\\\hline
$(7,5,2)/(4,1)$ & $[12x-3,12x+1]$
\\\hline
$(6,5,2,1)/(4,1)$ & $2x^2+6x+1$
\\\hline
$(6,4,3,1)/(3,2)$ & $[12x-3,16x+1]$
\\\hline
$(5,5,3,1)/(4,1)$ & $16x-7$
\\\hline
$(6,4,3,2)/(3,2,1)$ & $[14x-5,14x+1]$
\\\hline
$(6,5,4,1)/(4,3)$ & $2x^2+6x+1$
\\\hline
$(6,5,3,2)/(4,2,1)$ & $[12x-3,12x+1]$
\\\hline
$(5,5,2,2,2)/(4,1,1,1)$ & $4x^2+4x+1$
\\\hline
\end{longtable}

\vspace{-3mm}
\section*{Acknowledgements}\label{sec:acknowledge}

We would like to thank Pavel Etingof for suggesting this project and giving valuable advice along the way. We are also grateful to Dave Benson and Peter Symonds for helpful discussions. This research was conducted while the first author was a participant in the MIT PRIMES-USA program, which we thank for making this research opportunity possible. Research of K. B. Vashaw was partially supported by a National Science Foundation Postdoctoral Fellowship DMS-2103272.

\vspace{-1mm}
\section*{Declarations} \label{sec:decl}
The authors declare that they have no conflict of interest.

\end{document}